\documentclass[aos,preprint]{imsart}

\RequirePackage[OT1]{fontenc}
\RequirePackage{amsthm, amsmath}
\RequirePackage[numbers]{natbib}
\RequirePackage[colorlinks,citecolor=blue,urlcolor=blue]{hyperref}
\usepackage{color}

\usepackage{amssymb,bbm,bm}
\usepackage{algorithmic}
\usepackage[linesnumbered, ruled, vlined]{algorithm2e}
\usepackage{calrsfs}
\usepackage{mathrsfs}
\usepackage[T1]{fontenc}
\usepackage{appendix}
\usepackage{thmtools, thm-restate}
\usepackage{multirow}

\newcommand{\source}[1]{\hfill #1} 

\arxiv{arXiv:0000.0000}

\startlocaldefs

\theoremstyle{plain}
\newtheorem{definition}{Definition}[section]
\newtheorem{theorem}{Theorem}[section]
\newtheorem{corollary}{Corollary}[theorem]

\newtheorem{lemma}[theorem]{Lemma}
\newtheorem{assumption}{Assumption}
\newtheorem{remark}[theorem]{Remark}

\DeclareMathAlphabet{\pazocal}{OMS}{zplm}{m}{n}

\def\cD{{\pazocal D}}
\def\cK{{\pazocal K}}
\def\cH{{\pazocal H}}
\def\cX{{\pazocal X}}
\def\cU{{\pazocal U}}
\def\cV{{\pazocal V}}

\def\ccA{{\cal A}}
\def\ccR{{\cal R}}

\def\Node{{\pazocal A}}
\def\Noder{{\pazocal R}}

\def\mOne{{\mathbbm{1}}}
\def\pr{{\textup{pr}}}
\def\d{{\mathrm{d}}}

\def\bal#1\eal{\begin{align}#1\end{align}}

\def\tLambda{{\widetilde{\Lambda}}}
\def\hLambda{{\widehat{\Lambda}}}

\newcommand{\nc}{\nonumber\\}
\newcommand{\nn}{\nonumber}

\makeatletter
\newcommand{\pushright}[1]{\ifmeasuring@#1\else\omit\hfill$\displaystyle#1$\fi\ignorespaces}
\newcommand{\pushleft}[1]{\ifmeasuring@#1\else\omit$\displaystyle#1$\hfill\fi\ignorespaces}
\makeatother

\endlocaldefs

\begin{document}

\begin{frontmatter}
\title{Consistency of survival tree and forest models: splitting bias and correction\thanksref{T1}}
\runtitle{Splitting Bias and Consistency of Survival Forest}
\thankstext{T1}{This work is funded in part by NIH grant P01 CA142538 and NSF grant DMS-1407732}

\begin{aug}
\author{\fnms{Yifan} \snm{Cui} \thanksref{u1}\ead[label=e1]{cuiy@live.unc.edu}},
\author{\fnms{Ruoqing} \snm{Zhu} \thanksref{c1, u2}\ead[label=e2]{rqzhu@illinois.edu}}
\author{\fnms{Mai} \snm{Zhou} \thanksref{u3}\ead[label=e3]{mai@ms.uky.edu}}
\and
\author{\fnms{Michael R.} \snm{Kosorok}\thanksref{u1}
\ead[label=e4]{kosorok@unc.edu}}

\thankstext{c1}{The corresponding author}

\runauthor{Y. Cui et al.}

\affiliation{University of North Carolina at Chapel Hill \thanksmark{u1}, University of Illinois at Urbana-Champaign \thanksmark{u2}, University of Kentucky \thanksmark{u3}}

\address{Department of Statistics and Operations Research,\\
University of North Carolina at Chapel Hill,\\
Chapel Hill, NC 27599\\
\printead{e1}}

\address{Department of Statistics, \\
University of Illinois Urbana-Champaign, \\
Champaign, IL 61820\\
\printead{e2}}

\address{Department of Statistics,\\
University of Kentucky, Lexington, \\
KY 40506-0027\\
\printead{e3}}

\address{Department of Biostatistics,\\
University of North Carolina at Chapel Hill,\\
Chapel Hill, NC 27599\\
\printead{e4}}

\end{aug}

\begin{abstract}
Random survival forest and survival trees are popular models in statistics and machine learning. However, there is a lack of general understanding regarding consistency, splitting rules and influence of the censoring mechanism. In this paper, we investigate the statistical properties of existing methods from several interesting perspectives. First, we show that traditional splitting rules with censored outcomes rely on a biased estimation of the within-node failure distribution. To exactly quantify this bias, we develop a concentration bound of the within-node estimation based on non i.i.d. samples and apply it to the entire forest. Second, we analyze the entanglement between the failure and censoring distributions caused by univariate splits, and show that without correcting the bias at an internal node, survival tree and forest models can still enjoy consistency under suitable conditions. In particular, we demonstrate this property under two cases: a finite-dimensional case where the splitting variables and cutting points are chosen randomly, and a high-dimensional case where the covariates are weakly correlated. Our results can also degenerate into an independent covariate setting, which is commonly used in the random forest literature for high-dimensional sparse models. However, it may not be avoidable that the convergence rate depends on the total number of variables in the failure and censoring distributions. Third, we propose a new splitting rule that compares bias-corrected cumulative hazard functions at each internal node. We show that the rate of consistency of this new model depends only on the number of failure variables, which improves from non-bias-corrected versions. We perform simulation studies to confirm that this can substantially benefit the prediction error.
\end{abstract}

\begin{keyword}[class=MSC]
\kwd{62G20}
\kwd{62G08}
\kwd{62N01}
\end{keyword}

\begin{keyword}
\kwd{Random Forests}
\kwd{Survival Analysis}
\kwd{Consistency}
\kwd{Adaptive Concentration}
\kwd{Bias Correction}
\end{keyword}

\end{frontmatter}

\section{Introduction}\label{sec:intro}

Random forests \citep{breiman2001random} is one of the most popular machine learning tools. Many extensions of random forests \citep{rodriguez2006rotation, ishwaran2008random, chipman2010bart, menze2011oblique} have seen tremendous success in statistical and biomedical related problems \citep{liu2008cancer, rice2010cancer, bonato2010bayesian, green2012modeling, ofek2013improving, starling2014unexpected, huang2016longitudinal} in addition to many applications to artificial intelligence and other scientific fields.

The main advantage of forest- and tree-based \cite{breiman1984classification} models is their nonparametric nature. However, the theoretical properties have not been fully understood yet to date, even in regression settings, although there has been a surge of research on understanding the statistical properties of random forests in classification and regression. \cite{lin2006random} is one of the early attempts to connect random forests to nearest neighbor predictors. Later on, a series of works including \cite{biau2008consistency, biau2012analysis, genuer2012variance} and \cite{mentch2014ensemble} established theoretical results on simplified tree-building processes or specific aspects of the model. More recently, \cite{zhu2015reinforcement} established consistency results based on an improved splitting rule criteria; \cite{wager2014confidence} analyzed the confidence intervals induced from a random forest model; \cite{linero2016bayesian} established connections with Bayesian variable selection in the high dimensional setting; \cite{scornet2015consistency} showed consistency of the original random forests model on an additive model structure; and \cite{wager2015adaptive} studied the variance component of random forests and established corresponding concentration inequalities. For a more comprehensive review of related topics, we refer to \cite{biau2016random}.

In this paper, we focus on the theoretical properties of a particular type of tree- and forest-model, in which the outcomes are right-censored survival data \cite{fleming2011counting}. Censored survival data are frequently seen in biomedical studies when the actual clinical outcome may not be directly observable due to early dropout or other reasons. Random forest survival models have been developed to handle censored outcomes, including \citep{hothorn2004bagging, hothorn2005survival, ishwaran2008random, zhu2012recursively, steingrimsson2016doubly} and many others. However, there are few established theoretical results despite the popularity of these methods in practice, especially in genetic and clinical studies. For a general review of related topics, including single-tree based survival models, we refer to \cite{bou2011review}. To the best of our knowledge, the only consistency result to date is given in \cite{ishwaran2010consistency} who considered the setting where all predictors are categorical, while some other results are based on augmented outcomes which transform the problem to a fully observed regression model \cite{steingrimsson2017censoring}.

Our analysis provides insights into the consistency of survival forests and trees in general settings. In particular, we investigate whether existing methodologies enjoys consistency when covariates are weakly correlated if the splitting rule is searched by comparing the survival distribution of two potential child nodes. The answer is mixed because a biased selection of the splitting rule may occur if there are dependencies between the failure and censoring variables. This can, of course, be overcome in a finite dimensional case if the splitting rule is data-independent \cite{biau2012analysis, klusowski2018complete}, which leads to consistency. We further show that with suitable conditions, especially when the failure distribution signal size is sufficiently large, a marginal screening type of splitting rule can also lead to a consistent model. However, the consistency depends on the total dimension of variables that are involved in the failure and the censoring models.
Even when the covariates are uniformly distributed, as long as the failure and censoring times are not marginally independent, splitting on censoring covariates may not be avoidable due to entanglement between the failure and censoring distribution. This is mainly caused by univariate splits, and this understanding has never been formulated and analyzed in the literature. Motivated by this result, we propose a new bias-correction procedure that actively selects the best splitting variable at each internal node without the influence of the censoring distribution. This establishes a connection with existing methodology developments such as \cite{hothorn2005survival, steingrimsson2017censoring}, that converts censored observations to fully observed ones through the inverse probability of censoring weighting. However, our proposed splitting rule is much more general in the sense that it compares the distribution of survival time in the two potential child nodes, rather than focusing on the mean survival differences. We further show that with this new approach, we are able to improve the rate of consistency to depend on the dimension of the failure time model. Numerical results show that this yields a much-improved model fitting with reduced prediction error compared with its non-corrected counterparts.

The remainder of this paper is organized as follows to convey our findings: in Section \ref{sec:survtree}, we introduce tree-based survival models and some basic notation. Section \ref{sec:flaws} is devoted to demonstrating a fundamental property of the survival tree model associated with splitting rule selection and terminal node estimation. A concentration inequality of the Nelson-Aalen \cite{aalen1978nonparametric} estimator based on non-identically distributed samples is established. Utilizing this result, we derive adaptive concentration bounds of the entire model in Section \ref{sec:adaptive}. Furthermore, in Sections \ref{sec:consistency1} and \ref{sec:consistency2}, we establish consistency for two particular choices of splitting rules: a finite dimensional case with a random splitting rule and an infinite dimensional case with data adaptive marginal screening. Finally, in Section \ref{sec:biascorrect}, we propose a bias-corrected version of the random survival forest and compare that to the non-corrected versions and confirm our analysis. Further discussions are provided in Section \ref{sec:disc}. Details of proofs are given in the appendices, and a summary of notation is given before the appendices for convenience.

\section{Survival tree and forest models}\label{sec:survtree}

The essential element of tree-based survival models is recursive partitioning. A $d$-dimensional feature space ${\cX}$ is partitioned into terminal nodes. For a single tree model, we denote the collection of these terminal nodes as $\ccA = \{\Node_u\}_{u \in \cU}$, where $\cU$ is a set of indices, $\cX = \bigcup_{u \in \cU} \Node_u$ and $\Node_u \cap \Node_l = \emptyset$ for any $u\neq l$. We also call $\ccA$ a partition of the feature space $\cX$. In a traditional tree-building process \cite{breiman1984classification}, binary splitting rules are used. Hence all terminal node are (hyper)rectangles, i.e., $\Node = \bigotimes_{j=1}^d (a_j, b_j]$. In other examples, linear combination splits \citep{menze2011oblique, krketowska2004dipolar, zhu2015reinforcement} may result in more complex structures of terminal nodes. However, regardless of the construction of the trees, the terminal node estimations are obtained by treating the within-node observations as a homogeneous group. Before giving a general algorithm of tree-based survival models, we first introduce some notation.

Following the standard notation in the survival analysis literature, let $\cD_n = \{X_i, Y_i, \delta_i\}_{i=1}^n$ be a set of $n$ i.i.d. copies of the covariates, observed survival time, and censoring indicator, where the observed survival time $Y_i = \min(T_i, C_i)$, and $\delta_i = \mOne(T_i \leq C_i)$. We assume that each $T_i$ follows a conditional distribution $F_i(t) = \pr(T_i \leq t \mid X_i)$, where the survival function is denoted by $S_i(t) = 1 - F_i(t)$, the cumulative hazard function (CHF) $\Lambda_i(t) = - \log\{S_i(t)\}$, and the hazard function $\lambda_i(t) =\d\Lambda_i(t)/\d t$. The censoring time $C_i$ follows the conditional distribution $G_i(t) = \pr(C_i \leq t \mid X_i)$, where a non-informative censoring mechanism, $T_i \perp C_i \mid X_i$, is assumed.

In tree-based survival models, terminal node estimation is a crucial part. For any node $\Node_u$, this can be obtained through the Kaplan-Meier \citep{kaplan1958nonparametric} estimator for the survival function or the Nelson-Aalen \citep{nelson1969hazard, aalen1978nonparametric} estimator of the CHF based on the within-node samples. Our focus in this paper is on the following Nelson-Aalen estimator
\bal
\hLambda_{\Node_u,n}(t) =&~ {\textstyle\sum}_{s \leq t} \frac{\sum_{i=1}^n \mOne(\delta_i = 1)\mOne(Y_i = s)\mOne(X_i \in \Node_u)}{\sum_{i=1}^n \mOne(Y_i \geq s)\mOne(X_i \in \Node_u)}, \label{eqn:NA}
\eal
and the associated Nelson-Altshuler estimator \cite{altshuler1970theory} for the survival function when needed:
\bal
\widehat S_{\Node_u,n}(t) = \exp \big\{ - \hLambda_{\Node_u,n}(t) \big\}. \nonumber
\eal

A survival tree model yields a collection of doublets $\{\Node_u, \hLambda_{\Node_u,n}\}_{u \in \cU}$. In a survival forest model \citep{ishwaran2008random, zhu2012recursively}, a set of $B$ trees are fitted. In practice, $B = 1000$ is used in the popular \texttt{R} package \texttt{randomForestSRC} as the default value. Hence, a collection of partitions $\{ \{\Node_u^b, \hLambda_{\Node_u^b,n}\}_{u \in \cU_b} \}_{b = 1}^B$ is constructed. These trees are constructed with a bootstrap sampling mechanism or with the entire training data, with various types of additional randomness, such as \citep{geurts2006extremely}. To facilitate later arguments, we conclude this section by providing a high-level outline (Algorithm \ref{alg:forests}) of a survival forest model, while other details are deferred to later sections.

\begin{algorithm} 
\SetAlgoLined
\caption{Pseudo algorithm for survival forest models} \label{alg:forests}
\KwIn{Training dataset $\cD_n$, terminal node size $k$, number of trees $B$;}
\For{$b = 1$ \textbf{to} $B$}
{
\ShowLn Initiate $\Node = \cX$, a bootstrap sample $\cD_n^b$ of $\cD_n$, $\cK_b = \emptyset$, $u = 1$\;
\ShowLn At a node $\Node$, if $\sum_{X_i \in \cD_n^b} \mOne(X_i \in \Node) < k$, proceed to Line 5. Otherwise, construct a splitting rule such that $\Node = \Node_\text{left} \cup \Node_\text{right}$, where $\Node_\text{left} \cap \Node_\text{right} = \emptyset$. \;
\ShowLn Send the two child nodes $\Node_\text{left}$ and $\Node_\text{right}$ to Line 3 separately\;
\ShowLn Conclude the current node $\Node$ as a terminal node $\Node_u^b$, calculate $\widehat \Lambda_{\Node_u^b,n}$ using the within-node data, and update $\cK_b = \cK_b \cup \{u\}$ and $u = u+1$\;
}
\Return $\{\{\Node_u^b, \widehat \Lambda_{\Node_u^b,n}\}_{u \in \cK_b} \}_{b=1}^B$\
\end{algorithm}

\section{Biasedness of splitting rules}\label{sec:flaws}

One central idea throughout the survival tree and forest literature is to construct certain goodness-of-fit statistics that evaluate the impurity reduction across many candidate splitting rules. The best splitting rule is then selected and implemented to partition the node. This essentially resembles the idea in a regression tree setting where the mean differences or equivalently the variance reduction is used as the criterion. The most popular criterion in survival tree models is constructed through the log-rank statistic \citep{gordon1985tree, ciampi1986stratification, leblanc1992relative, fan2006trees, ishwaran2008random, zhu2012recursively} and other nonparametric comparisons of two curves, such as the Kolmogorov-Smirnov, Wilcoxon-Gehan and Tarone-Ware statistics \citep{ciampi1989recpam, segal1988regression}. However, these methods are generally believed to be biased because they do not appropriately consider the censoring distribution. Hence, other ideas are motivated, such as likelihood or model-based approaches \citep{ciampi1987recursive, davis1989exponential, loh1991survival, ahn1994tree, su2004multivariate, fan2006trees}, inverse probability of censoring weighting (IPCW) \citep{molinaro2004tree, hothorn2005survival}, double robust correction of the observed data \cite{steingrimsson2017censoring}, and others \cite{zhang1995splitting, krketowska2006random}. \cite{bou2011review} provides a comprehensive review of the methodological developments of survival tree models.

The literature on the theoretical analysis of survival tree and forest methods is relatively sparse. One of the more recent general results is \cite{ishwaran2010consistency}, who established uniform consistency of random survival forests \citep{ishwaran2008random} by assuming a discrete feature space, which can happen, for example, when the covariates are genotyping data. The idea can be extended to many other specific survival tree models. However, the discrete distribution assumption of the feature space is not satisfied in most applications. The major difficulty of the theoretical developments in a general setting is the highly complex nature of the splitting rules and their interference with the entire tree structure.

\subsection{Within-node estimation\label{sec:withinnode}}

To begin our analysis, we start by investigating the Kaplan-Meier (KM) and the Nelson-Altshuler (NA) estimators of the survival function. There are two main reasons that we revisit these classical methods: first, these methods are wildly used for terminal node estimation in fitted survival trees. Hence, the consistency of any survival tree model inevitably relies on their asymptotic behavior; second, the most popular splitting rules, such as the log-rank, Wilcoxon-Gehan and Tarone-Ware statistics, are all essentially comparing the KM curves across the two potential child nodes, which again plays an important role in the consistency results. We note that although other splitting criteria exist, our theoretical framework can be extended to their particular settings. Without making restrictive distributional assumptions on the underlying model, our results show that these popular splitting rules, not surprisingly, are affected by the underlying censoring distribution. Hence the splitting variables are essentially biased, in the sense that the most important variable may not be used asymptotically. Our contribution to the existing literature is that we exactly quantify this biased estimator by developing a concentration bound around its true mean. Noticing that the KM and the NA estimators are the two most commonly used estimators, the following Lemma bounds their difference through an exact inequality regardless of the underlying data distribution. The proof follows mostly from \cite{cuzick1985asymptotic} and is given in the Appendices.

\begin{lemma}\label{lem:KM.NA.diff}
Let $\widehat S_{\scriptscriptstyle K\!M}(t)$ and $\widehat S_{\scriptscriptstyle N\!A}(t)$ be the Kaplan-Meier and the Nelson-Altshuler estimators, respectively, obtained using the same set of samples $\{Y_i, \delta_i\}_{i=1}^n$. Then we have,
\bal
\lvert \widehat S_{\scriptscriptstyle K\!M}(t) - \widehat S_{\scriptscriptstyle N\!A}(t) \rvert < \widehat S_{\scriptscriptstyle K\!M}(t) \frac{4}{\sum_{i=1}^n \mOne(Y_i \geq t)}, \nonumber
\eal
for any observed failure time point $t$ such that $\widehat S_{\scriptscriptstyle K\!M}(t)> 0$.
\end{lemma}
The above result suggests that calculating the difference between two KM curves is asymptotically the same as using the NA estimator as long as we only calculate the curve up to a time point where the sample size is sufficiently large. For this purpose, Assumption \ref{asm:tau} is always used throughout this paper. Note that similar assumptions are commonly used in the survival analysis literature, for examples, $\pr(T \geq \tau) > 0$ in \cite{fleming2011counting}, and $\pr(C \geq \tau) > 0$ in \cite{murphy1997maximum}. Then, with large probability, $\widehat S_{\scriptscriptstyle N\!A}(t) = \widehat S_{\scriptscriptstyle K\!M}(t) + O(1/n)$ across all terminal nodes.

\begin{assumption}\label{asm:tau}
There exists fixed positive constants $\tau < \infty$ and $M \in (0, 1)$, such that
\bal
\pr(Y \geq \tau \mid X=x) \geq M, \nn
\eal
for all $x \in \cX$.
\end{assumption}

\subsection{A motivating example}\label{sec:toy}

Noticing that the splitting rule selection process is essentially comparing the survival curves computed from two potential child nodes, we take a closer look at this process. Most existing analyses of the KM estimator assume that the observations are i.i.d. \citep{breslow1974large, gill1980censoring} or at least one set of the failure times or censoring times are i.i.d. \cite{zhou1991some}. However, this is almost always not true for tree-based methods at any internal node because both $T_i$'s and $C_i$'s typically depend on the covariates. The question is whether this affects the selection of the splitting variable. We first use a simulation study to demonstrate this issue.

Consider the split at a particular node. We generate three random variables: $X^{(1)}$, $X^{(2)}$ and $X^{(3)}$ from a multivariate normal distribution with mean 0 and covariance matrix $\Sigma$, where all diagonal elements of $\Sigma$ are 1, and the nonzero off-diagonal elements are $\Sigma_{12} = \Sigma_{21} = 0.8$. The failure distribution is exponential with mean $\exp(-1.25 X^{(1)} - X^{(3)} + 2)$. We consider two censoring distributions for $C$: an exponential distribution with mean 2 for all subjects, i.e., independent of $T$; and the second one is an exponential distribution with mean equal to $\exp(-3 X^{(2)})$. The splitting rule is searched for by maximizing the log-rank test statistic between the two potential child nodes $\{X^{(j)} \leq c, X \in \Node\}$ and $\{X^{(j)} > c, X \in \Node\}$, and the cutting point $c$ is searched on the range of the variable. In an ideal situation, one would expect the best splitting rule to be constructed using $X^{(1)}$ with large probability, since it carries the most signal. This is indeed the case as shown in the first row of Table \ref{tab:prob} for the i.i.d. censoring case, but not so much for the dependent censoring case. The simulation is done with $n = 1000$ and repeated 1000 times. While this only demonstrates the splitting process on a single node, the consequence of this on the consistency of the entire tree is much more involved since the entire tree structure can be altered by the censoring distribution. It is difficult to draw a definite conclusion at this point, but the impact of the censoring distribution is clearly demonstrated.

\begin{table}[htp]\caption{Probability of selecting the splitting variable.}\label{tab:prob}
\begin{center}
    \begin{tabular}{|l|c|c|c|}
    \hline
    \noalign{\smallskip}
       Censoring distribution & $X^{(1)}$ & $X^{(2)}$ & $X^{(3)}$ \\
        \noalign{\smallskip}
        \hline
        \noalign{\smallskip}
    $G_i$ identical & 0.978 & 0.001 & 0.021 \\
    $G_i$ depends on $X_i^{(2)}$ & 0.281 & 0.037 & 0.682 \\
    \noalign{\smallskip}
    \hline
    \end{tabular}
\end{center}
\end{table}

\subsection{Survival estimation based on independent, non-identically distributed observations\label{sec:withinbias}}
It now seems impossible to analyze the consistency without exactly quantifying the within node estimation performance. We look at two different quantities corresponding to the two scenarios used above. The first one is an averaged CHF within any node $\Node$:
\bal\label{eqn:nodemeanLambda}
\Lambda_\Node(t) = \frac{1}{\mu(\Node)}\int_{x \in \Node} \Lambda(t \mid x) \textup{d} \textup{P}(x),
\eal
where $\textup{P}$ is the distribution of $X$, and $\mu(\Node)= \int_{x \in \Node} \textup{d} \textup{P}(x) $ is the measure of node $\Node$. Clearly, since in the first case, the censoring distribution is not covariate dependent, we are asymptotically comparing $\Lambda_\Node(t)$ on the two child nodes, which results in the selection of the first variable. This should also be considered as a rational choice since $X^{(1)}$ contains more signal at the current node.

In the second scenario, i.e., the dependent censoring case, the within-node estimator $\hLambda_{\Node,n}(t)$ does not converge to the $\Lambda_\Node(t)$ in general, which can be inferred from the following theorem. As the main result of this section, Theorem \ref{thm:survbound} is interesting in its own right for understanding tree-based survival models, since it establishes a bound of the survival estimation under independent but non-identically distributed samples, which is a more general result than \cite{zhou1991some}. It quantifies exactly the estimation performance for each potential child node, hence is also crucial for understanding splitting rules in general.

\begin{theorem}\label{thm:survbound}
Let $\hLambda_n(t)$ be the Nelson-Aalen estimator of the CHF from a set of $n$ independent samples $\{Y_i, \delta_i\}_{i=1}^n$ subject to right censoring, where the failure and censoring distributions (not necessarily identical) are given by $F_i$'s and $G_i$'s. Under Assumption \ref{asm:tau}, we have for $\epsilon_1 \leq 2$ and $n > 4 / (\epsilon_1^2 M^4)$,
\bal
\pr\bigg( \underset{t < \tau}{\sup} \left\lvert \hLambda_n(t) - \Lambda_n^\ast(t) \right\rvert > \epsilon_1 \bigg) < 16(n+2) \exp\Big\{\frac{-n M^4 \epsilon_1^2}{1152}\Big\},
\eal
where
\bal
\Lambda_n^\ast(t) = \int_0^t \frac{\sum[1-G_i(s)] \textup{d} F_i(s)}{\sum[1-G_i(s)][1 - F_i(s)]}. \label{Lambdastar}
\eal
\end{theorem}

The proof is deferred to Appendix \ref{sec:lamdastar}. Based on Theorem \ref{thm:survbound}, if we restrict ourselves to any node $\Node$, the difference between the within-node estimator $\hLambda_{\Node, n}(t)$ and
\bal
\Lambda_{\Node,n}^\ast(t) = \int_0^t \frac{\sum_{X_i \in \Node}[1-G_i(s)] \textup{d} F_i(s)}{\sum_{X_i \in \Node}[1-G_i(s)][1 - F_i(s)]} \label{nodemeanLambdastar}
\eal
is bounded above, where $\Lambda_{\Node,n}^\ast(t)$ is some version of the underlying true cumulative hazard contaminated by the censoring distribution. Noting that $\Lambda_{\Node,n}^\ast(t)$ also depends on the sampling points $X_i$'s, we further develop Lemma \ref{lemma:sample/expectation0} in the Appendix to verify that $\Lambda^\ast_{\Node, n}(t)$ and its expected version $\Lambda^\ast_\Node(t)$ are close enough, where
\begin{align}
\Lambda_\Node^\ast(t)=\int_0^t \frac{E_{X\in\Node} [1-G(s\mid X)] \textup{d} F(s\mid X)}{ E_{X\in\Node} [1-G(s\mid X)][1 - F(s\mid X)]}.
\label{nodetrueLambdastar}
\end{align}
It is easy to see that the difference between $\Lambda_{\Node,n}^\ast(t)$ and $\Lambda_\Node(t)$ will vanish if the $F_i$'s are identical within a node $\Node$ (a sufficient condition). Note that this is what we are hoping for eventually at a terminal node.
\bal
\Lambda_{\Node, n}^\ast(t) =& \int_0^t \frac{\sum_{X_i \in \Node}[1-G_i(s)]}{\sum_{X_i \in \Node}[1-G_i(s)]} \frac{\textup{d} F(s)}{1 - F(s)} \nc
& \pushright{(\text{if $F_i \equiv F$ for all $X_i \in \Node$})} \nc
=&~ \int_0^t \frac{\textup{d} F(s)}{1 - F(s)} = \frac{1}{\mu(\Node)}\int_{x \in \Node} \int_0^t \frac{\textup{d} F(s)}{1 - F(s)} \textup{d} \textup{P}(x) = \Lambda_\Node(t).
\eal

As we demonstrated in the simulation study above, comparing $\hLambda_{\Node,n}(t)$ between two child nodes may lead to a systematically different selection of splitting variables than using $\Lambda_\Node(t)$ which is not known a priori. The main cause of the differences between these two quantities is that the NA estimator treats each node as a homogeneous group, which is typically not true. Another simple interpretation is that although the conditional independence assumption $T \perp C \mid X$ is satisfied, we have instead $T \not\perp C \,\mid \,\mOne(X^{(j)} < x)$ at an internal node. This tangling between the censoring and failure distributions makes it a very changeling problem because, at each internal node, we may only select one variable to split.

\section{Consistency of survival tree and forest}\label{sec:consistency}

It becomes apparent now that this bias plays an important role in the asymptotic properties of survival tree and forest models. However, an important question we may ask is, whether this affects the consistency of existing methodologies? To answer this question, we provide several analyses of consistency in different settings. Noticing that difficulty arises when the splitting rule is highly data dependent, we first investigate the consistency under random splitting rules and finite $d$ (Section \ref{sec:consistency1}) to help our understanding. An analog of this result for regression and classification settings was proposed by \cite{breiman2004consistency}, and further analyzed by \cite{lin2006random, biau2008consistency, biau2012analysis, arlot2014analysis} and many others. However, it is significantly more difficult when the splitting rule is data dependent, especially when the number of dimensions $d$ is diverging with $n$. Note that $\hLambda_n(t)$ is a biased estimation of the within node averaged CHF, any marginal comparison splitting rule may falsely select a variable involved in the censoring mechanism. Furthermore, any dependencies in the high-dimensional setting may lead us to further falsely select a noise variable because they will marginally carry signals. Hence, we investigate a high-dimensional setting where the noise variables have weak dependencies with the important variables (to be defined later). Under suitable conditions and some necessary modifications to the tree-build process, we show that a survival tree or forest model can still achieve consistency, however, with its rate possibly affected by the censoring distribution. To establish these results, we use the variance-bias breakdown and start by analyzing the variance component of a survival tree estimator.

\subsection{Adaptive concentration bounds of survival trees and forests}\label{sec:adaptive}

We focus on quantifying the survival forest models from a new angle, namely the adaptive concentration \citep{wager2015adaptive} of each terminal node estimator to the true within-node expectation. In the sense of the variance-bias breakdown, the goal of this section is to quantify a version of the variance component of a tree-based model estimator. To be precise, with large probability, our main results bound $\big\lvert \hLambda_{\Node,n}(t) - \Lambda_{\Node, n}^\ast(t) \big\rvert$ across all possible terminal nodes in a fitted tree or forest. The adaptiveness comes from the fact that the target of the concentration is the censoring contaminated version $\Lambda_{\Node, n}^\ast(t)$, which is adaptively defined for each node $\Node$ with the observed samples, rather than as a fixed universal value. The results in this section have many implications. Since this bound is essentially the variance part of the estimator, we can then focus the bias part to show consistencies. Although this may still pose challenges in specific situations, our later examples of consistency provide a framework that is largely applicable to most existing methods.

We start with some additional definitions and notations. Following our previous assumptions on the underlying data generating model, we observe a set of $n$ i.i.d. samples $\cD_n$. We view each tree as a partition of the feature space, denoted $\ccA = \{\Node_u\}_{u \in \cU}$, where the $\Node_u$'s are non-overlapping hyper-rectangular terminal nodes. We first define valid survival tree and forest estimators of the CHF. Roughly speaking, with certain constraints, these are all the possible survival tree or forest estimators resulted from a set of observed data. First, borrowing Definitions in \cite{wager2015adaptive}, we denote the set of all $\{\alpha,k\}$-valid tree partitions by $\cV_{\alpha,k}(\cD)$. In addition, we define the collection $\{\ccA^{(b)}\}_{b=1}^B$ as a valid forest partition if each of its tree partitions are valid. We denote the set of all such valid forest partitions as $\cH_{\alpha,k}(\cD)$. A valid survival tree or forest estimator is induced from the corresponding valid set.

\begin{definition}[Valid survival tree and forest]
Given the observed data $\cD_n$, a valid survival tree estimator of the CHF is induced by a valid partition $\ccA \in \cV_{\alpha,k}(\cD_n)$ with $\ccA = \{\Node_u\}_{u \in \cU}$:
\bal
\hLambda_{\ccA,n}(t \mid x) = \sum_{u \in \cU} \mOne(x \in \Node_u) \hLambda_{\Node_u,n}(t),
\eal
where each $\hLambda_{\Node_u,n}(t \mid x)$ is defined by Equation \eqref{eqn:NA}. Furthermore, a valid survival forest $\hLambda_{\{\ccA_{(b)}\}_1^B,n}$ is defined as the average of $B$ valid survival trees induced by a collection of valid partitions $\{\ccA_{(b)}\}_1^B \in \cH_{\alpha,k}(\cD_n)$,
\bal
\hLambda_{\{\ccA_{(b)}\}_1^B,n}(t \mid x) = \frac{1}{B}\sum_{b=1}^B \hLambda_{\ccA_{(b)},n}(t\mid x).
\eal
\end{definition}

We also define the censoring contaminated survival tree and forest, which are asymptotic versions of the corresponding within-node average estimators of the CHF. Note that by Theorem \ref{thm:survbound}, these averages are censoring contaminated versions $\Lambda^\ast_{\Node, n}(t)$, but not the true averages $\Lambda_\Node(t)$.

\begin{definition}[Censoring contaminated survival tree and forest] Given the observed data $\cD_n$ and $\ccA \in \cV_{\alpha,k}(\cD_n)$, the corresponding censoring contaminated survival tree is defined as
\bal
\Lambda_{\ccA, n}^\ast(t \mid x) = \sum_{u \in \cU} \mOne(x \in \Node_u) \Lambda_{\Node_u, n}^\ast (t),
\eal
where each $\Lambda_{\Node_u, n}^\ast (t)$ is defined by Equation \eqref{nodemeanLambdastar}. Furthermore, let $\{\ccA_{(b)}\}_1^B \in \cH_{\alpha,k}(\cD_n)$. Then the censoring contaminated survival forest is given by
\bal
\Lambda_{\{\ccA_{(b)}\}_1^B, \,n}^\ast(t \mid x) = \frac{1}{B}\sum_{b=1}^B \Lambda_{\ccA_{(b)}, \,n}^\ast (t\mid x).
\eal
\end{definition}

Our adaptive concentration bound result considers the quantity
\begin{align*}
\hLambda_{\ccA,n}(t\mid x)- \Lambda_{\ccA, n}^\ast(t\mid x)
\end{align*}
for all valid partitions $\ccA \in \cV_{\alpha,k}(\cD_n)$. We first specify several regularity assumptions. The first assumption is a bound on the dependence of the individual features. Note that in the literature, uniform distributions are often assumed \cite{biau2008consistency, biau2012analysis} on the covariates, which implies independence. To allow dependency among covariates, we assume the following two assumptions, which have also been considered in \cite{wager2015adaptive}. Without loss of generality, we assume that covariates are distributed on $[0,1]^d$.
\begin{assumption}
 Covariates $X\in [0,1]^d$ are distributed according to a density function $p(\cdot)$ satisfying $1/\zeta \leq p(x) \leq \zeta$ for all $x$ and some $\zeta\geq 1$.
\label{asm:density}
\end{assumption}
We also set a restriction on the tuning parameter $k$, the minimum terminal node size, which may grow with $n$ and dimension $d$ via the following rate:
\begin{assumption}
Assume that $k$ is bounded below so that
\begin{equation}
\lim_{n \rightarrow \infty} \frac{\log(n) \max\{\log(d),\log\log(n)\}}{k}=0.
\end{equation}\label{asm:k}
\end{assumption}

Then we have the adaptive bound for our tree estimator in the following theorem. The proof is collected in Appendix \ref{proof:thm:bound}.

\begin{theorem}\label{thm:bound}
Suppose the training data $\cD_n$ satisfy Assumptions \ref{asm:tau} and \ref{asm:density}, and the rate of the sequence $(n,d,k)$ satisfies Assumption \ref{asm:k}. Then all valid trees concentrate on a censoring contaminated tree:
\begin{align*}
\sup_{t<\tau,\,x\in [0,1]^d,\,\ccA \in \cV_{\alpha,k}(\cD_n) } & \left\lvert \hLambda_{\ccA,n}(t\mid x)-\Lambda_{\ccA, n}^\ast(t\mid x) \right\rvert \\
 \leq&~ M_1\sqrt {  \frac{\log(n/k)[\log(dk)+\log\log(n)]}{k\log((1-\alpha)^{-1})} },
\end{align*}
with probability larger than $1-2/\sqrt n$, for some universal constant $M_1$.
\end{theorem}

The above theorem holds for all single tree partitions in $\cV_{\alpha,k}(\cD_n)$. Consequently, we have a similar result for the forest estimator in Corollary~\ref{thm:boundforest} in Appendix \ref{proof:thm:bound}.

\begin{remark}
In a moderately high dimensional setting, i.e. $d\sim n$, the rate is $\log (n)/k^{1/2}$. In an ultra high dimensional setting, for example, $\log(d) \sim n^{\vartheta}$, where $0<\vartheta<1$, the rate is close to $n^{\vartheta}/k^{1/2}$. The rate that $k$ grows with $n$ cannot be too slow in order to achieve the bound in the ultra high dimensional setting. This is somewhat intuitive since if $k$ grows too slowly then we are not able to bound all possible nodes in a single tree.
\end{remark}

The results established in this section essentially address the variance component in a fitted random forest. We chose not to use the true within-node population averaged quantity $\Lambda_\Node^\ast(t)$ (see Equation \ref{nodetrueLambdastar}), or its single tree and forest versions as the target of the concentration. This is because such a result would require bounded density function of the failure time $T$. However, when $f(t)$ is bounded, the results can be easily generalized to $\big\lvert \hLambda_{\Node,n}(t) - \Lambda_{\Node}^\ast(t) \big\rvert$. Lemma \ref{lemma:sampleexpectation} in Section \ref{sec:consistency2} provides an analog of Theorem \ref{thm:survbound} in this situation.

With the above concentration inequalities established, we are now in a position to discuss consistency results. We consider two specific scenarios: a finite dimensional case where the splitting rule is generated randomly, and an infinite dimensional case using the marginal difference of Nelson-Aalen estimators as the splitting rule.

\subsection{Consistency under random splitting rules when dimension $d$ is finite}\label{sec:consistency1}

Assume that the dimension $d$ of the covariates space is fixed and finite. At each internal node we choose the splitting variable randomly and uniformly from all covariates \cite{biau2012analysis, klusowski2018complete}. When the splitting variable is chosen, we choose the splitting point uniformly at random such that both child nodes contain at least $\alpha$ proportion of the samples in the parent node. We will bound the bias term
\begin{align*}
\sup_{t<\tau} E_X \left\lvert\Lambda_{\{\ccA_{(b)}\}_1^B, \, n}^\ast(t\mid X)-\Lambda (t\mid X)\right\rvert.
\end{align*}
It should be noted that in Section \ref{sec:adaptive}, we did not treat the tree and forest structures ($\ccA$ and $\{\ccA_{(b)}\}_1^B$) as random variables. Instead, they were treated as elements of the valid structure sets. However, in this section, once a particular splitting rule is specified, these structures become random variables associated with certain distributions induced from the splitting rule. When there is no risk of ambiguity, we inherit the notation $\hLambda_{\ccA,n}$ to represent a tree estimator, where the randomness of $\ccA$ is understood as part of the randomness in the estimator itself. A similar strategy is applied to the forest version. Before presenting the consistency results, we make an additional smoothness assumption on the hazard function:

\begin{assumption}
For any fixed time point $t$, the CHF $\Lambda(t\mid x)$ is $L_1$-Lipschitz continuous in terms of $x$, and the hazard function $\lambda(t\mid x)$ is $L_2$-Lipschitz continuous in terms of $x$, i.e., $|\Lambda(t \mid x_1)-\Lambda(t \mid x_2)| \leq  L_1||x_1-x_2||$ and $|\lambda(t \mid x_1)-\lambda(t \mid x_2)|\leq  L_2\|x_1-x_2\|$, respectively, where $\|\cdot\|$ is the Euclidean norm.
\label{asm:lips}
\end{assumption}

We are now ready to state our main consistency results for the proposed survival tree model. Theorem \ref{treeconsistency1} provides the point-wise consistency result. The proof is collected in Appendix \ref{proof:consistency1}.
\begin{theorem}\label{treeconsistency1}
Under Assumptions \ref{asm:tau}--\ref{asm:lips}, the proposed survival tree model with random splitting rule is consistent, i.e., for each $x\in [0,1]^d$,
\begin{align*}
\sup_{t<\tau} \big| \hLambda_{\ccA,n}(t\mid x)-\Lambda (t\mid x) \big| = O\bigg (\sqrt { \frac{\log(n/k)[\log(dk)+\log\log(n)]}{k\log((1-\alpha)^{-1})} }+\Big(\frac{k}{n}\Big)^{\frac{c_1}{d}}\bigg ),
\end{align*}
with probability at least $1-w_n$, where
\begin{align*}
w_n=\frac{2}{\sqrt{n}}+d \exp\Big\{-\frac{c_2^2\log_{1/\alpha}(n/k)}{2d}\Big\}+d \exp\Big\{- \frac{(1-c_2)c_3c_4^2 \log_{1/\alpha}(n/k)}{2d}\Big\},
\end{align*}
and $c_2,c_4\in(0,1)$, $c_3=(1-2\alpha)/8$ and $c_1={c_3(1-c_2)(1-c_4)}/{\log_{1-\alpha}(\alpha)}$.
\end{theorem}

\begin{remark}
The first part $\sqrt{ \frac{\log(n/k)[\log(dk)+\log\log(n)]}{k\log((1-\alpha)^{-1})} }$ in the bound comes from the concentration results and the second part $({k}/{n})^{{c_1}/{d}}$ comes from the bias. We point out that the optimal rate is obtained by setting $k=n^{{c_3}/{[c_3+d\log_{1-\alpha}(\alpha)/2]}}$, and then the optimal rate is close to $n^{-{c_3}/{[2c_3+d\log_{1-\alpha}(\alpha)]}}$. If we further assume that we always split at the middle point at each internal node, then the optimal rate degenerates to $n^{-{1}/{(d+2)}}$ obtained by setting $k=n^{{d}/{(d+2)}}$.
\label{remark}
\end{remark}

The consistency result can be easily extended to survival forests with $B$ trees. Theorem \ref{survivalforestcons1.1} presents an integrated version, which can be derived from Theorem \ref{treeconsistency1}. The proof is collected in Appendix \ref{proof:consistency1}.

\begin{theorem}\label{survivalforestcons1.1}
Under Assumptions \ref{asm:tau}-\ref{asm:lips}, the proposed survival forest is consistent, i.e.,
\begin{align*}
&\lim_{B\rightarrow \infty} \sup_{t<\tau} E_X|\hLambda_{\{\ccA_{(b)}\}_1^B,n}(t\mid X)-\Lambda(t\mid X)| \\ \nonumber
=&~ O \bigg(\sqrt {  \frac{\log(n/k)[\log(dk)+\log\log(n)]}{k\log((1-\alpha)^{-1})} }+\Big(\frac{k}{n}\Big)^{\frac{c_1}{d}}+\log(k)w_n \bigg),
\end{align*}
where $w_n$ is a sequence approaching to 0 as defined in Theorem \ref{treeconsistency1}, $c_2,c_4\in (0,1)$, $c_3=(1-2\alpha)/8$ and $c_1=c_3(1-c_2)(1-c_4)/\log_{1-\alpha}(\alpha)$.
\end{theorem}


\subsection{Consistency under adaptive splitting rules}\label{sec:consistency2}

We have so far established consistency results for both survival trees and forests for finite dimension $d$. In this section, we allow the dimension $d$ go to infinity with sample size $n$ while the covariates are possibly correlated. We note that this is not a commonly used setting in the literature because it can be difficult to control the marginal distribution of a noise variable. We first provide several definitions. We assume that there are $d_0$ important features involved in the failure time distribution, and denote $\mathcal M_F$ the set of their indices, hence $\mathcal M_F \subset \{1, \ldots d\}$. Precisely, we have $T \perp X | X_{\mathcal M_{F}}$. The number of features involved in both the failure time and censoring time distributions is $d_0 + d_1$, where $d_0$ and $d_1$ are fixed, and denoted $\mathcal M_{FC}$ their indices. However, this does not imply that there are only $d_1$ variables for the censoring distribution, because the failure variables and censoring variables may share some common indices. For example, age, as a commonly used demographic variable can be informative for both the clinical outcome of interest (failure) and the lost to follow-up (censoring). However, splitting such variables is necessary as long as they are involved in the failure distribution. The splits that we may want to avoid is on the variables in the set $\mathcal M_C$ defined as $\mathcal M_{FC} \setminus \mathcal M_F$, making $d_1 = |\mathcal M_C|$. However, our later analysis indicates that this may not be avoidable because of the biasedness caused by censoring. Lastly, we define the set of noise variables' indices as $\mathcal M_N = \{1, \ldots d\} \setminus \mathcal M_{FC}$. We make the following assumption on the weak dependencies between noise variables and variables in $\mathcal M_{FC}$.

\begin{assumption}
We assume that the conditional distribution of the failure and censoring covariates $X_{\mathcal M_{FC}}$  has weak dependencies on any univariate noise variable, in the sense that for some constant $\gamma>1$, we have
$$\gamma^{-1} <\frac{p(X_{\mathcal M_{FC}}=x|X^{(j)}=x_1)}{p(X_{\mathcal M_{FC}}=x|X^{(j)}=x_2)}< \gamma,$$
for any $x_1,x_2\in [0,1]$, $(d_0+d_1)$-dimensional vector $x$ and any $j\in \mathcal M_{N}$.
\label{asm:cor}
\end{assumption}

Assumption \ref{asm:cor} relaxes the commonly used independent covariate assumption in the literature \cite{biau2012analysis, zhu2015reinforcement, scornet2015consistency, doi:10.1080/01621459.2017.1319839}. However, this poses significant difficulties when evaluating the marginal signal carried by a noise variable, meaning that the difference between the two potential child nodes may not be 0. A large threshold is necessary to prevent the noise variables from entering the model, as shown in the later results. However, as the correlation reduces to 0, i.e., $\gamma = 1$, the threshold will naturally degenerate to 0. We further need an assumption on the effect size of the failure variable.

\begin{assumption}\label{asm:censor} Marginal signal of the failure distribution. Let $\Node$ be any internal node, and $j \in {\mathcal M_F}$ be a variable that has never been split, i.e., the range of $X^{(j)}$ in node $\Node$ is $[0, 1]$. Let $\Node^+_j$ and $\Node^-_j$ be defined as $\Node_j^+(c)=\{ X: X^{(j)} \geq c\}$ and $\Node_j^-(c)=\{ X: X^{(j)} < c\}$, respectively. And further define the right/left average CHFs as
\bal
\ell^+(j,t,c)&= \int_0^t \frac{E_{X\in \Node^+_j (c)}f(s\mid X)}{E_{X\in \Node^+_j (c)} [1-F(s\mid X)]} \, \textup{d}s, \nonumber \\
\ell^-(j,t,c)&= \int_0^t \frac{E_{X\in \Node^-_j (c)} f(s\mid X)}{E_{X\in \Node^-_j (c)}  [1-F(s\mid X)]} \, \textup{d}s. \nonumber
\eal
Then there exists a time point $t_0 \in [0, \tau]$, a constant $c_0\in (0,1)$, and minimum effect size $\ell>2(\gamma^2-\gamma^{-2} ) {\tau L}/{M^2}$, such that
\bal
 & {M} \ell^+(j,t_0,c_0) - {M}^{-1} \ell^-(j,t_0,c_0) \, > \ell,  \quad \text{if} \quad   \ell^+(j,t_0,c_0) > \ell^-(j,t_0,c_0), \nonumber \\
\text{or} \,\,\,\,  & {M}  \ell^-(j,t_0,c_0) - {M}^{-1} \ell^+(j,t_0,c_0) \, > \ell,  \quad \text{if}  \quad  \ell^+(j,t_0,c_0) < \ell^-(j,t_0,c_0), \nonumber
\eal
where $\tau$ and $M$ are defined in Assumption~\ref{asm:tau}, $\gamma$ is defined in Assumption~\ref{asm:cor}, and $L$ is an upper bound of $f(t|x)$ for all $x \in [0,1]^d$.
\end{assumption}

This assumption can be interpreted as the following. First, $\ell^+(j, t, c)$ and $\ell^-(j, t, c)$ are the averaged CHFs on the left and right-hand side, respectively, of a split at the $j$th variable. The constant $M$ and its reciprocal can be understood as the minimum and maximum contamination, respectively, of the censoring distribution on these CHFs. The assumption requires that if the difference of these contaminated versions is sufficiently large at some time point $t_0$ and some cut point $c_0$. Note that $ M$ is a lower bound of $\pr(C \geq \tau \mid X = x)$, this essentially bounds below the signal size regardless of any dependency structures between $C$ and $T$. However, in some trivial cases, such as when the $G_i$'s are identical, the constant $M$ can be removed from the assumption due to the independence between $T$ and $C$. A simplified version will be provided in Assumption \ref{asm:censor2} in Section \ref{sec:biascorrect}. Furthermore, $\ell$ can be an arbitrarily small constant if $\gamma = 1$, which is essentially the independent covariates case.

Another important observation of this assumption is that $t_0$ can be arbitrary. Hence, we essentially allow the CHFs of different subjects to cross each other. As a comparison, we note that in many popular survival models, such as the Cox proportional hazard model, the CHF is a monotone function of $X$ on the entire time domain. Hence the survival curve of any subject can only be completely above or below that of another subject. However, when the survival curves cross each other, a log-rank test may not be effectively \cite{fleming2011counting, eng2005sample}. When we incorporate the splitting rule (in Algorithm \ref{alg:consistency2}) that detects the maximum differences on $[0, \tau)$. Hence, our model is capable of detecting non-monotone signals of the CHF as a function of both $X$ and $t$, making our approach more powerful than the traditional log-rank test splitting rule.

\begin{algorithm}[ht] 
\SetAlgoLined
\caption{A marginal splitting rule for survival forest} \label{alg:consistency2}
\ShowLn At any internal node $\Node$ containing at least $2k$ training samples, we pick a splitting variable $j\in\{1,\ldots,d\}$ uniformly at random\;
\ShowLn We then pick the splitting point $\tilde c$ using the following rule such that both child nodes contain at least proportion $\alpha$ of the samples at $\Node$:
\begin{align*}
\tilde c= \underset{c}{\arg\max} \,\, \Delta_1(c),
\end{align*}
where $\Delta_1(c)=\max_{t<\tau} \big|\hLambda_{\Node_j^+(c),n}(t)-\hLambda_{\Node_j^-(c),n}(t) \big|$,
$\Node_j^+(c)=\{ X: X^{(j)} \geq c\}$, and $\Node_j^-(c)=\{ X: X^{(j)} < c\}$, $X^{(j)}$ is the $j$-th dimension of $X$\;
\ShowLn If the variable $j$ has already been used along the sequence of splitting rules leading up to $\Node$, or the following inequality holds for some constant $M_3$:
\begin{align*}
\Delta_1(\tilde c)\geq (\gamma^2-\gamma^{-2} ) \frac{\tau L}{M^2}+M_3  \sqrt {  \frac{\log(n/k)[\log(dk)+\log\log(n)]}{k\log((1-\alpha)^{-1})} },
\end{align*}
then we split at $\tilde c$ along the $j$-th variable. If not, we randomly sample another variable out of the remaining variables and proceed to Step 2). When there is no remaining feasible variable, we randomly select an index out of $d$ and split at $\tilde c$.
\end{algorithm}

Finally, to make a splitting rule concrete, we provide Algorithm \ref{alg:consistency2}, which marginally compares the estimated CHF over all time points and uses the difference to select the best split. Based on this algorithm, Lemma \ref{lemma:splitcombine2} in Appendix~\ref{proof:consistency2} shows that our $d$ dimensional survival forest is equivalent to a $(d_0+d_1)$ dimensional survival forest with probability larger than $1-3/\sqrt{n}$. This means that with a large probability, we shall never split on the noise variable set $\mathcal{M}_N$. We highlight here that $\Lambda_{\Node, n}^\ast(t)$ is an essential tool to prove Lemma \ref{lemma:splitcombine2}. The intuition here is that when the failure distribution doesn't depend on the variable $j$, the quantity

\begin{align*}
&\int_0^t \frac{E_{X\in \Node_j^+(x)} [1-G(s\mid X)] \textup{d} F(s\mid X)}{E_{X\in \Node_j^+(x)} [1-G(s\mid X)][1 - F(s\mid X)]} \\
&\qquad \qquad \qquad  -\int_0^t \frac{E_{X\in \Node_j^-(x)} [1-G(s\mid X)] \textup{d} F(s\mid X)}{E_{X\in \Node_j^-(x)} [1-G(s\mid X)][1 - F(s\mid X)]}
\end{align*}
is bounded by a small constant under weak dependency. However, this is quantity will degenerate to 0 as long as the dependency vanishes, i.e., $\gamma = 1$, since $dF(s|X)/[1-F(s|X)]$ separates regardless of what the censoring distribution is. The proof is indeed beautiful and neat, which is deferred to Appendix~\ref{proof:consistency2}.

Notice that $\Lambda_{\Node, n}^\ast(t)$ is a sample version of the asymptotic distribution of the terminal node $\Node$. In Lemma \ref{lemma:sampleexpectation}, we show the bound of the difference of $\Lambda_{\Node, n}^\ast(t)$ and its integrated version $\Lambda_\Node^\ast(t)$ across all valid nodes $\Node$, where $\Lambda_\Node^\ast(t)$ is as defined in Equation \eqref{nodetrueLambdastar}. The proof is given in Appendix \ref{proof:consistency2}.

%



\begin{lemma}%
Under Assumptions~\ref{asm:tau}-\ref{asm:k} and further assume that the conditional density function $f(t \mid x)$ of the failure time $T$ is bounded by $L$ for all $x\in [0,1]^d$. The difference between $\Lambda_{\ccA, n}^\ast(t)$ and $ \Lambda_\ccA^\ast(t)$ is bounded by
\begin{align*}
\sup_{t<\tau,\,x\in [0,1]^d,\,\ccA \in \cV_{\alpha,k}(\cD_n) } & \left\lvert \Lambda_{\ccA, n}^\ast(t\mid x)- \Lambda_\ccA^\ast(t\mid x) \right\rvert \\\nonumber
\leq&~ M_2\sqrt {  \frac{\log(n/k)[\log(dk)+\log\log(n)]}{k\log((1-\alpha)^{-1})} },
\end{align*}
with probability larger than $1-1/\sqrt n$.
\label{lemma:sampleexpectation}
\end{lemma}

Based on Lemma \ref{lemma:splitcombine2} provided in Appendix~\ref{proof:consistency2}, we essentially only split on $(d_0+d_1)$ dimensions with probability larger than $1-3/\sqrt{n}$ on the entire tree. The consistency holds from Theorem \ref{treeconsistency1}. The following result shows the consistency of the proposed survival forest. The proof is almost identical to Theorem \ref{survivalforestcons1.1}:


\begin{theorem}\label{survivalforestcons2}
Under Assumptions \ref{asm:tau}-\ref{asm:censor}, the proposed survival tree using the splitting rule specified in Algorithm \ref{alg:consistency2} is consistent, i.e., for any $x$, 
\begin{align*}
& \sup_{t<\tau} |\hLambda_{\{\ccA_{(b)}\}_1^B,n}(t\mid x)-\Lambda(t\mid x)| \\ \nonumber
=&~ O \bigg(\sqrt {  \frac{\log(n/k)[\log\{(d_0+d_1)k\}+\log\log(n)]}{k\log((1-\alpha)^{-1})} }+\Big(\frac{k}{n}\Big)^{\frac{c_1}{d_0+d_1}}\bigg),
\end{align*}
with probability at least $1-w_n$, where
\begin{align*}
w_n=\frac{3}{\sqrt{n}}+(d_0+d_1)\Big[\exp\Big\{-\frac{c_2^2\log_{1/\alpha}(n/k)}{2(d_0+d_1)}\Big\}+ \exp\Big\{- \frac{(1-c_2)c_3c_4^2 \log_{1/\alpha}(n/k)}{2(d_0+d_1)}\Big\}\Big],
\end{align*}
$c_2,c_4\in (0,1)$, $c_3=(1-2\alpha)/8$ and $c_1=c_3(1-c_2)(1-c_4)/\log_{1-\alpha}(\alpha)$. Consequently, the proposed survival forest is consistent, i.e.,
\begin{align*}
&\lim_{B\rightarrow \infty} \sup_{t<\tau} E_X|\hLambda_{\{\ccA_{(b)}\}_1^B,n}(t\mid X)-\Lambda(t\mid X)| \\ \nonumber
=&~ O \bigg(\sqrt {  \frac{\log(n/k)[\log\{(d_0+d_1)k\}+\log\log(n)]}{k\log((1-\alpha)^{-1})} }+\Big(\frac{k}{n}\Big)^{\frac{c_1}{d_0+d_1}}+\log(k)w_n \bigg).
\end{align*}
\end{theorem}

Although we have developed a result where $d$ can grow exponentially fast with $n$ in this section, the splitting rule implemented was not completely the same as the practically used version because it essentially checks only the signal where the candidate variables have never been used. This is made possible with Assumption \ref{asm:censor} that the difference between two potential child nodes resulted from $X^{(j)} < c_0$ versus $X^{(j)} \geq c_0$ is sufficiently large. Once a variable is used, it will be automatically included as a candidate in subsequent splits. The idea is similar to the protected variable set used in \cite{zhu2015reinforcement}, where the protected set serves as the collection of variables that have been used in previous nodes.

\begin{remark}
We provide a consistency result for survival trees and forests under weak dependency framework. We highlight that our results hold naturally if $X$ is uniformly distributed with $\zeta=\gamma=1$. Notice that when variables are uncorrelated, our results are still meaningful and constructive: the biasedness is not due to correlated variables but the entanglement between failure time and censoring time marginally. For example, if censoring time shares one common variable with failure time, then all other censoring variables also play a role in the limiting distribution, so there is no guarantee to split only on failure variables.
\end{remark}

\section{A bias-corrected survival forest}\label{sec:biascorrect}

\subsection{Methodology\label{sec:method}}

Recall that in Section~\ref{sec:flaws}, we investigated the estimation bias in the comparison of two potential child nodes. This is caused by ignoring the within-node heterogeneity of the censoring distribution while it is entangled with the failure distribution. A closer look at Equation~\eqref{nodemeanLambdastar} motivates us to perform a weighted version of the Nelson-Aalen estimator to correct this bias and estimate the true within-node averaged CHF. Hence, we consider
\bal
\tLambda_{\Node,n}(t) =&~ {\textstyle\sum}_{s \leq t} \frac{\sum_{i=1}^n \mOne(\delta_i = 1)\mOne(Y_i = s)\mOne(X_i \in \Node)/[1-\widehat G(s|X_i)]}{\sum_{i=1}^n \mOne(Y_i \geq s)\mOne(X_i \in \Node)/[1-\widehat G(s|X_i)]}, \label{eqn:wNA}
\eal
where $\widehat G(s|X_i)$ is an estimated conditional censoring distribution function. Note that this estimator resembles a form of the inverse probably weighting strategy \cite{rotnitzky2005inverse}, which is an extensively studied technique in the survival analysis and missing data literature \cite{robins1992recovery}. There have been many different forms of inverse probably weighted estimators under a variety of contexts. For example, \cite{hothorn2005survival} uses $\delta_i /(1-\widehat G(Y_i|X_i))$ as the subject-specific weight to fit regression random forests. One can also transform the censored observations into fully observed ones using, e.g., \cite{rubin2007doubly}, and then fit a regression model with the complete data \cite{molinaro2004tree, steingrimsson2016doubly, steingrimsson2017censoring}. Similar ideas have also been used for imputing censored outcomes \cite{zhu2012recursively} when learning an optimal treatment strategy \cite{cui2017tree}.

However, our proposal is fundamentally different from these existing methods. We estimate and compare an inverse probably weighted hazard function using weight $\delta_i /(1-\widehat G(s|X_i))$ and repeatedly performing this at each internal node. A key observation is that the comparison is over the entire domain of the survival time instead of fitting regression forests based on complete observations. This is a unique advantage because the distribution function contains richer information than the within-node means. This makes our approach more sensitive for detecting differences between two potential child nodes at all quantile levels of the survival time. It also advocates the goal of a typical survival analysis model, where the survival function is the target of interest rather than the expected survival time. The intuition has a close connection with the sup-log-rank test statistic \cite{fleming2011counting, eng2005sample}, which can be used to detect any distributional difference of $T$ of the two potential child nodes. Furthermore, with the following modified algorithm, we can achieve an improved convergence rate that depends only the size of $\mathcal{M}_F$.

Algorithm~\ref{alg:consistency2} can be modified accordingly to incorporate this new procedure. In particular, at each internal node $\Node$, we use the weighted CHF estimator $\tLambda_{\Node,n}(t)$ defined in Equation~\eqref{eqn:wNA}. We then pick the splitting point $\tilde c$ with the rule such that both child nodes contain at least proportion $\alpha$ of the samples at $\Node$:
\begin{align*}
\tilde c= \underset{c}{\arg\max} \,\, \Delta_2(c),
\end{align*}
where $\Delta_2(c)=\max_{t<\tau} \big|\tLambda_{\Node_j^+(c),n}(t)-\tLambda_{\Node_j^-(c),n}(t) \big|$, $\Node_j^+(c)=\{ X: X^{(j)} \geq c\}$, and $\Node_j^-(c)=\{ X: X^{(j)} < c\}$, $X^{(j)}$ is the $j$-th dimension of $X$.

%

Note that the threshold of $\Delta_2(c)$ in this bias-corrected version is the same as the one used for $\Delta_1(c)$ in Algorithm~\ref{alg:consistency2}. The intuition is that after removing the censoring bias, variables in $\mathcal{M}_{C}$ play the same role as noise variables in $\mathcal{M}_{N}$. In addition, the signal size Assumption~\ref{asm:censor} can be relaxed to the following.

\begin{assumption}\label{asm:censor2} Marginal signal of the failure distribution. Let $\ell^+(j, t_0, c_0)$, $\ell^-(j, t_0, c_0)$ and effect size $l$ be as defined in Assumption~\ref{asm:censor}. Then, there exists a time point $t_0$ and a cutting point $c_0$ such that, for any $j \in {\mathcal M_F}$,
\begin{align*}
\bigg| \ell^+(j, t_0, c_0) - \ell^-(j, t_0, c_0) \bigg| > \ell.
\end{align*}
\end{assumption}
Note that this is essentially removing the censoring contaminated part ($M$ and its reciprocal) from Assumption~\ref{asm:censor}. Of course, this is at the cost of plugging-in a consistent estimator of the censoring distribution $G$ to correct the bias. On the other hand, we need an additional assumption on the dependency structures.
\begin{assumption}
We assume that the conditional distribution of the failure and censoring covariates $X_{\mathcal M_{FC}}$  has weak dependencies on any univariate censoring variable, in the sense that for constant $\gamma>1$, we have
$$\gamma^{-1} <\frac{p(X_{\mathcal M_{FC}\setminus \{j\}}=x|X^{(j)}=x_1)}{p(X_{\mathcal M_{FC}\setminus \{j\} }=x|X^{(j)}=x_2)}< \gamma,$$
for any $x_1,x_2\in [0,1]$, $(d_0+d_1-1)$-dimensional vector $x$ and any $j\in \mathcal M_{C}$.
\label{asm:cor2}
\end{assumption}
This is an analog of Assumption~\ref{asm:cor} to further prevent the censoring variables from carrying strong marginal signals due to correlations. It will again degenerate to the commonly used independent covaraite case when $\gamma = 1$. Finally, we show that consistency can be established based on our new model fitting procedure, with convergence rate depends only on the number of variables in $\mathcal{M}_F$. Details of the proof are collected in the Appendix.

\begin{theorem}\label{thm:survivalforestcons3}
Under Assumptions \ref{asm:tau}-\ref{asm:cor},  \ref{asm:censor2} and \ref{asm:cor2}, assuming that $\widehat G$ in Equation~ \eqref{eqn:wNA} is a consistent estimation of the censoring distribution, the proposed bias-corrected survival forest is consistent, i.e., for any $x$,
\begin{align*}
& \sup_{t<\tau} |\hLambda_{\{\ccA_{(b)}\}_1^B,n}(t\mid x)-\Lambda(t\mid x)| \\ \nonumber
=&~ O \bigg(\sqrt {  \frac{\log(n/k)[\log(d_0k)+\log\log(n)]}{k\log((1-\alpha)^{-1})} }+\Big(\frac{k}{n}\Big)^{\frac{c_1}{d_0}}\bigg),
\end{align*}
with probability at least $1-w_n$, where
\begin{align*}
w_n=\frac{3}{\sqrt{n}}+d_0\exp\Big\{-\frac{c_2^2\log_{1/\alpha}(n/k)}{2d_0}\Big\}+ d_0\exp\Big\{- \frac{(1-c_2)c_3c_4^2 \log_{1/\alpha}(n/k)}{2d_0}\Big\},
\end{align*}
$c_2,c_4\in (0,1)$, $c_3=(1-2\alpha)/8$ and $c_1=c_3(1-c_2)(1-c_4)/\log_{1-\alpha}(\alpha)$. Consequently, the proposed survival forest is consistent, i.e.,
\begin{align*}
&\lim_{B\rightarrow \infty} \sup_{t<\tau} E_X|\hLambda_{\{\ccA_{(b)}\}_1^B,n}(t\mid X)-\Lambda(t\mid X)| \\ \nonumber
=&~ O \bigg(\sqrt {  \frac{\log(n/k)[\log(d_0k)+\log\log(n)]}{k\log((1-\alpha)^{-1})} }+\Big(\frac{k}{n}\Big)^{\frac{c_1}{d_0}}+\log(k)w_n \bigg).
\end{align*}
\end{theorem}

\subsection{Numerical results\label{sec:simu}}

To fully understand the impact of bias correction, we consider a set of simulation studies. There are many existing implementations of random survival forests, including R packages \texttt{randomForestSRC} \cite{ishwaran2019random}, \texttt{party} \cite{hothorn2006unbiased}, \texttt{ranger} \cite{wright2017ranger}, etc. However, it is difficult to compare across different packages as they may utilize certain tuning parameters slightly differently. It would not be possible to investigate the sole impact of bias correction if these subtle differences are involved. Hence, we turn to make our own implementation of survival forest modeling with and without the bias correction, while ensuring all other mechanisms remain exactly the same.

Furthermore, we note that there are two possible ways to make a bias correction based on our previous analysis. First, and most apparently, we can incorporate $\tLambda_{\Node,n}$ in $\Delta_2(c)$ to search for a better splitting rule. Alternatively, we may apply a regular splitting rule and use $\tLambda_{\Node,n}$ only in the terminal node estimation to correct the bias. Based on our analysis, the second approach would not improve the convergence rate because the tree structure is already built on $\mathcal{M}_{FC}$ variables, while the first approach has the potential for improvement. Hence, contrasting these two approaches would allow us to investigate the importance of changing the entire tree structure through bias correction. We consider four different algorithms out of the combination of these two choices: 1) (C-C) bias-corrected splitting rule and terminal node estimation; 2) (C-N) bias-corrected splitting rule without correcting the terminal node estimation; 3) (N-C) correcting only the terminal node estimation; and 4) (N-N) do not perform any bias correction. Note again that we implement all four methods under the same algorithm framework that assures all other tuning parameters remain the same.

We consider two data generating scenarios, each with dependent censoring and independent censoring mechanisms. For the first scenario, we consider the setting in Section~\ref{sec:toy}. Let $d=3$ and $X^{(1)}$, $X^{(2)}$ and $X^{(3)}$ from a multivariate normal distribution with mean 0 and variance $\Sigma$, where the diagonal elements of $\Sigma$ are all 1, and the only nonzero off diagonal element is $\Sigma_{12} = \Sigma_{21} = \rho= 0.8$. $T$ is exponential distribution with mean $\exp(-1.25X^{(1)} - X^{(3)}+ 2)$. For dependent censoring, the censoring time follows an exponential distribution with mean $\exp(-3X^{(2)})$; For independent censoring, the censoring time follows an exponential distribution with mean $2$. For the second scenario, we consider a setting where the covariates are independent. we let $d = 10$ and draw $X$ from a multivariate normal distribution with mean 0. 
Survival times are drawn independently from an accelerated failure time model, $\log(T)=X^{(1)}+X^{(2)}+X^{(3)}+\epsilon$, where $\epsilon$ is generated from a standard normal distribution. For dependent censoring, the censoring time shares the variable $X^{(3)}$ with the failure time $T$, and follows $\log(C)=-1+2X^{(3)}+X^{(4)}+X^{(5)}+\epsilon$; For marginal independent censoring, the censoring time follows $\log(C)=-1+X^{(4)}+X^{(5)}+2X^{(6)}+\epsilon$.


For each setting, we use a training dataset with sample size $n = 400$. A testing dataset with size $800$ is used to evaluate the mean squared error of the estimated conditional survival functions \cite{zhu2012recursively}. Each simulation is repeated 500 times. Note that since all methods are implemented under the same code, we fix the tuning parameters and only investigate the influence of bias correction. Tuning parameters in the survival trees are chosen as follows.  According to \cite{ishwaran2008random}, the number of covariates considered at each splitting is set to $\lceil \sqrt{d} \, \rceil$. The minimal number of observed failures in each terminal node is set to 10 and 25, respectively such that $k \approx n^{d/(d+2)}$ as pointed out in Remark~\ref{remark}. We set the total number of trees to be 100.


The simulation results are summarized in Table~\ref{tab1}. In both scenarios, we clearly see that the bias-corrected splitting rule (C-C and C-N columns) has significantly improved the performance compared with N-C and N-N. This shows that by selecting a better variable to split, the tree structure can be corrected to reduce the prediction accuracy. For scenario~2, which has a more complicated censoring structure, the improvement of the proposed bias-corrected splitting rule is more significant than Scenario~1, with the average mean square error decreasing approximately from 47.7 to 42. Hence, the performance of bias correction may also depend on the complexity of the censoring distribution and the accuracy of its estimation. The standard error is comparable among all four methods.

Interestingly, we want to highlight that biasedness is mainly caused by the splitting bias rather than terminal node estimation, i.e., C-C is similar to C-N, and N-C is similar to N-N. This is intuitive and is in line with our theory that the splitting rule bias-correction procedure can enjoy potentially faster convergence rate than the non-bias-corrected version. One would not expect a good prediction if trees are partitioned inefficiently regardless of what kind of terminal node estimation is being used. After the tree is constructed, there is not much room to correct the bias if previous splits were chosen on noise or censoring variables.

\begin{table}[h]
\caption{\label{tab1}
Simulation results: Mean and (standard deviation) of mean square error}
\begin{tabular}{cccccc}
\noalign{\smallskip}
 & Censoring & C-C & C-N & N-C & N-N  \\
\noalign{\smallskip}
\multirow{2}{*}{Scenario 1} & Dependent  &  21.62~(7.46) & 21.68 ~(7.50) & 23.32 ~(7.18) & 23.29 ~(7.17)\\
                            & Independent & ~8.42 ~(2.01) & ~8.41 ~(1.98) & ~8.42 ~(2.02) & ~8.43 ~(2.04) \\

\noalign{\smallskip}
\multirow{2}{*}{Scenario 2} & Dependent  & 42.02~(5.37) & 41.97 ~(5.34) & 47.76 ~(5.76) & 47.75 ~(5.77)\\
                            & Independent & 35.18~(3.80) & 35.22 ~(3.77) & 36.11 ~(3.55) & 36.14 ~(3.59) \\
\noalign{\smallskip}
\end{tabular}
\source{\label{tab1} C-C/C-N/N-C/N-N refer to the configurations of correcting or not correcting the bias, while the first letter refers to the splitting rule correction, and the second letter refers to the terminal node correction.}
\end{table}

\section{Discussion}\label{sec:disc}

In this paper, we provided insights into tree- and forest-based survival models and developed several fundamental results analyzing the impact of splitting rules. We first investigated the within-node Nelson-Aalen estimator of the CHF and established a concentration inequality for independent but non-identically distributed samples. By introducing a new concept called the censoring contamination, we can exactly quantify the bias of existing splitting rules. Based on this, we also developed a concentration inequality that bounds the variance component of survival trees and forests. We further analyzed in detail how such bias can affect the consistency. In particular, we show that for a commonly used marginal comparison splitting rule strategy, the convergence rate depends on the total number of variables involved in the failure and censoring distribution. However, by appropriately correcting the bias, the convergence rate depends only on the number of failure variables. Essentially, the new bias correction procedure can be understood as untangling the failure and censoring distributions.

In addition to analyzing this entanglement, our result is based on a weak dependency structure, which bounds the marginal signal of any noise variable. This is a generalization of the commonly used independent covariate setting in the literature. A univariate split has a disadvantage when dealing with noise variables because if they are systematically selected in the splitting rule, the convergence rate will suffer. We believe that similar weak dependency assumptions are inevitable, because otherwise, any correlation structure may carry signals into the noise variables. It would be interesting to investigate whether more advanced splitting rules can overcome this drawback.

\appendix
\newpage

The following table provides a summary of Appendices. \\

{\scriptsize
\begin{tabular}{| p{0.95in} | p{3.5in} |}
\hline
Appendix \ref{sec:lamdastar} & Proof of Theorem \ref{thm:survbound} that bounds $|\hLambda_n(t) - \Lambda_n^\ast(t)|$. The proof utilizes three technic lemmas \ref{lemma:A1}, \ref{lemma:A2} and \ref{lemma:A3}. We also provide Lemma \ref{lem:KM.NA.diff} which bounds the difference between the Nelson-Altshuler and the Kaplan-Meier estimators. \\
Appendix \ref{proof:thm:bound} & Proof of the adaptive concentration bound Theorem \ref{thm:bound} and its forest version, Corollary \ref{thm:boundforest}.\\
Appendix \ref{proof:consistency1} & Proof of the consistency when the dimension $d$ is fixed: the point-wise error bound (Theorem \ref{treeconsistency1}), and an integrated version (Theorem \ref{survivalforestcons1.1}).\\
Appendix \ref{proof:consistency2} & Proof of consistency (Theorem \ref{survivalforestcons2}) when using the nonparametric splitting rule defined in Algorithm \ref{alg:consistency2}.\\
Appendix \ref{proof:consistency3} & Proof of consistency (Theorem \ref{thm:survivalforestcons3}) for the proposed bias-corrected survival tree and forest models.\\
\hline
\end{tabular}
}\\
\\

The following table provides a summary of notation used in the proofs. \\

{\scriptsize
\begin{tabular}{| p{0.90in} | p{3.6in} |}
\hline
\multicolumn{2}{|l|}{Basic Notation}\\
\hline
$T$ & Failure time \\
$C$ & Censoring time \\
$Y$ & $= \min(T, C)$: observed time \\
$\delta$ & $= \mOne(T \leq C)$: censoring indicator\\
$F_i$, $f_i$ & Survival distribution of $i$-th observation, $f_i = dF_i$ \\
$G_i$ & Censoring distribution of $i$-th observation \\
$\Node$ & A node, internal or terminal \\
$\ccA$ & $= \{\Node_u\}_{u \in \cU}$, the collection of all terminal nodes in a single tree \\
$\Lambda(t|x)$ & Cumulative hazard function (CHF) \\
$\hLambda_n$, $\hLambda_{\Node,n}$, $\hLambda_{\ccA,n}$ & NA estimator on a set of samples, a node $\Node$, or an entire tree $\ccA$\\
$\Lambda_n^\ast$, $\Lambda_{\Node,n}^\ast$, $\Lambda_{\ccA,n}^\ast$ & Censoring contaminated averaged CHF on a set of samples, a node $\Node$, or the entire tree $\ccA$\\
$\Lambda^\ast$, $\Lambda_\Node^\ast$, $\Lambda_\ccA^\ast$ & Population versions of $\Lambda_n^\ast$, $\Lambda_{\Node,n}^\ast$ and $\Lambda_{\ccA,n}^\ast$, respectively\\
$\tLambda_n$, $\tLambda_{\Node,n}$, $\tLambda_{\ccA,n}$ & Biased correct NA estimator on a set of samples, a node $\Node$, or an entire tree $\ccA$\\
$k$ & 
Minimum leaf size\\
$\alpha$ & Minimum proportion of observations contained in child node\\
$B$ &  Number of trees in a forest\\
$\cV_{\alpha,k}(\cD)$ & Set of all $\{\alpha,k\}$ valid partitions on the feature space $\cX$\\
$\cH_{\alpha,k}(\cD)$ & Set of all $\{\alpha,k\}$ valid forests on the feature space $\cX$\\
$\Noder$ & Approximation node \\
$\ccR$ & The set of approximation nodes\\
$N(t)$ & Counting process \\
$K(t)$ &  At-risk process \\
$\mu(\Noder)$, $\mu(\Node)$ & The expected fraction of training samples inside $\Noder, \Node$  \\
$\#\Noder$, $\#\Node$  & The number of training samples inside $\Noder, \Node$ \\
$\mathcal M_{F}, \mathcal M_{C}, \mathcal M_{N}$ &  Set of indices of failure variables, censoring variables, noise variables, respectively \\
\hline
\multicolumn{2}{|l|}{Constants}\\
\hline
$d(d_0,d_1)$ &  Dimension of (failure, censoring) covariates\\
$\tau$ & The positive constant as the upper bound of $Y$\\
$M$ & Lower bound of $\pr(Y \geq \tau | X)$ \\
$\zeta$ & A constant used in Assumption~\ref{asm:density}\\
$L$ & Bound of the density function $f(t)$\\
$L_1$, $L_2$ & Lipschitz constant of $\Lambda$ and $\lambda$\\
$\ell, \ell'$ & Minimum effect size of marginal signal of the failure distribution \\
$\gamma$ & Bound for weak dependency\\
\hline
\end{tabular}
}

\section{}\label{sec:lamdastar}

\noindent {\bf Proof of Lemma~\ref{lem:KM.NA.diff}.} For simplicity, we prove the results for the case when there are no ties in the failure time. The proof follows mostly \cite{cuzick1985asymptotic}. Let $n_1 > n_2 > \ldots > n_k \geq 1$ be the sequence of counts of the at-risk sample size, i.e., $n_j = \sum_{i=1}^n \mOne(Y_i \geq t_j)$, where $t_i$ is the $i$th ordered failure time. Then the Kaplan-Meier estimator at any observed failure time point $t_j$ can be expressed as $\widehat S_{\scriptscriptstyle K\!M} (t_j) = \prod_{i=1}^j (n_i - 1)/n_i$, while the Nelson-Altshuler estimator at the same time point is $\widehat S_{\scriptscriptstyle N\!A}(t_j) = \exp\{ - \sum_{i=1}^j 1/n_i \}$. We first apply the Taylor expansion of $e^{-n_i}$ for $n_i \geq 1$:
\bal
1 - 1/n_i < e^{-n_i} < 1 - 1/n_i + 1 / (2 n_i^2) \leq 1 - 1/(n_i + 1). \nonumber
\eal
Thus we can bound the Nelson-Altshuler estimator with
\bal
\widehat S_{\scriptscriptstyle K\!M}(t_j) < \widehat S_{\scriptscriptstyle N\!A}(t_j) < \textstyle \prod_{i=1}^j n_i/(n_i + 1). \nonumber
\eal
To bound the difference between the two estimators, note that for $n_j \geq 2$,
\bal
\left\lvert \widehat S_{\scriptscriptstyle K\!M}(t_j) - \widehat S_{\scriptscriptstyle N\!A}(t_j) \right\rvert <&~ \left\lvert \widehat S_{\scriptscriptstyle K\!M}(t_j) - \textstyle \prod_{i=1}^j n_i/(n_i + 1) \right\rvert \nc
=&~ \widehat S_{\scriptscriptstyle K\!M}(t_j) \left\lvert 1 - \textstyle \prod_{i=1}^j \frac{n_i/(n_i + 1)}{(n_i - 1)/n_i}  \right\rvert \nc
\leq&~ \widehat S_{\scriptscriptstyle K\!M}(t_j) \textstyle \sum_{i=1}^j (n_i^2 - 1)^{-1} \nc
\leq&~ 2 \widehat S_{\scriptscriptstyle K\!M}(t_j) \textstyle \sum_{i=1}^j n_i^{-2} \nc
\leq&~ 4 \widehat S_{\scriptscriptstyle K\!M}(t_j) /n_j.
\eal
Now note that both the Kaplan-Meier and the Nelson-Altshuler estimators stay constant within $(t_i, t_{i+1})$, and this bound applies to the entire interval $(0, t_k)$ for $n_k \geq 2$. $\Box$\\

\noindent {\bf Proof of Theorem~\ref{thm:survbound}.} Recall the counting process
\begin{align*}
N(s)=\sum_{i=1}^{n} N_i(s)=\sum_{i=1}^{n} \mOne(Y_i\leq s, \delta_i=1),
\end{align*}
and the at risk process
\begin{align*}
K(s)=\sum_{i=1}^{n} K_i(s)=\sum_{i=1}^{n} \mOne(Y_i\geq s).
\end{align*}

We prove the theorem based on the following key results. 
\begin{lemma}\label{lemma:A1}
Provided Assumption \ref{asm:tau} holds, for arbitrary $\epsilon>0$
and $n$ such that $\frac{1}{n}\leq \frac{\epsilon^2}{2}$, we have
\begin{align*}
\pr(\sup_{t\leq \tau} \lvert \frac{1}{n} \sum_{i=1}^{n} \{K_i(s)-E[K_i(s)]\} \rvert> \epsilon )\leq 8(n+1)\exp\big\{ \frac{-n\epsilon^2}{32} \big\},\\
\pr(\sup_{t \leq \tau} |\frac{1}{n} \sum_{i=1}^{n} \{N_i (t)-E[ N_i (t)]\} |> \epsilon )\leq 8(n+1)\exp\big\{ \frac{-n\epsilon^2}{32} \big\}.
\end{align*}
\end{lemma}

\begin{lemma}\label{lemma:A2}
Provided Assumption \ref{asm:tau} holds, for any $\epsilon >0$, we have
\begin{align*}
\pr(\sup_{t\leq \tau} |\int_0^t  (\frac{1}{K(s)}-\frac{1}{E[K(s)]})  dN(s)|> \epsilon) \leq 8(n+2)\exp\Big\{ -\frac{n\min(\epsilon^2 M^4, M^2)}{128} \Big\},
\end{align*}
where $n$ satisfies $\frac{1}{n}<\min(\frac{\epsilon^2}{2},\frac{\epsilon^2M^4}{4})$ and $M$ is defined in Assumption \ref{asm:tau}.
\end{lemma}

\begin{lemma}\label{lemma:A3}
Provided Assumption \ref{asm:tau} holds, for any $\epsilon >0$, we have
\begin{align*}
\pr(\sup_{t\leq \tau} |\int_0^t  \frac{d\{N(s)-E[N(s)]\}}{E[K(s)]}|> \epsilon) \leq 8(n+1)\exp\Big\{-\frac{n\epsilon^2 M^2}{228} \Big\},
\end{align*}
where $n$ satisfies $\frac{1}{n}\leq \frac{\epsilon^2}{2}$ and $M$ is defined in Assumption \ref{asm:tau}.
\end{lemma}
The proof of Lemma \ref{lemma:A1} follows pages 14--16 in \cite{pollard2012convergence}. The proofs of Lemma \ref{lemma:A2} and \ref{lemma:A3} are presented below. Now we are ready to prove Theorem~\ref{thm:survbound}. Note that
\begin{align*}
&~\pr\big(\sup_{t < \tau} \lvert \hLambda_n(t) - \Lambda_n^\ast(t) \rvert >\epsilon_1 \big) \\
=&~ \pr\Big(\sup_{t < \tau} \lvert \hLambda_n(t) - \int_0^{t}\frac{dE[N(s)]}{E[K(s)]} \rvert >\epsilon_1\Big)  \\
\leq &~ \pr\Big(\sup_{t\leq \tau}  \Big| \int_0^t \Big[\frac{1}{K(s)}-\frac{1}{E[K(s)]}\Big] dN(s)\Big| >\frac{\epsilon_1}{2} \Big) \\
& + \pr\Big( \sup_{t\leq \tau}  \Big| \int_0^t \frac{d\{N(s)-E[N(s)]\}}{E[K(s)]} \Big| > \frac{\epsilon_1}{2} \Big).
\end{align*}

By Lemma \ref{lemma:A2}, the first term is bounded by $8(n+2)\exp{ \{-\frac{n \min(\epsilon_1^2 M^4,4M^2)}{512} \}}$. By Lemma \ref{lemma:A3}, the second term is bounded by $8(n+1)\exp{ \{-\frac{n\epsilon_1^2 M^2}{1152} \}}$. The sum of these two terms is further bounded by $16(n+2)\exp{ \{- \frac{n\epsilon_1^2 M^4}{1152}\}}$ for any $\epsilon_1\leq 2$ and $n> \frac{4}{\epsilon_1^2 M^4}$. This completes the proof. $\Box$\\

\noindent {\bf Proof of Lemma~\ref{lemma:A2}.} For any $t\leq \tau$,
\begin{align}
&\Big|\int_0^t  (\frac{1}{K(s)}-\frac{1}{E[K(s)]})  dN(s) \Big| \nonumber \\
\leq & \int_0^t \frac{|E[K(s)]-K(s)|}{K(s)E[K(s)]}dN(s) \nonumber \\
\leq & \int_0^t  \frac{ \underset{0<r\leq \tau}{\sup} |E[K(r)]-K(r) | }{K(s)E[K(s)]}dN(s).
\label{eq:atrisk}
\end{align}

Thanks to Hoeffding's inequality, we have
\begin{align*}
\pr\Big( \big|K(\tau)-E[K(\tau)] \big| > \frac{nM}{2} \Big)< 2\exp\big\{ -\frac{nM^2}{2} \big\}.
\end{align*}
Then \eqref{eq:atrisk} is further bounded by
\begin{align*}
\frac{n}{(nM)^2/2} \,\, \underset{0<t\leq \tau}{\sup} \, \big|E[K(t)]-K(t) \big|.
\end{align*}
Combining with Lemma \ref{lemma:A1}, we have
\begin{align*}
& \pr \Big(\sup_{t\leq \tau} \Big|\int_0^t  (\frac{1}{K(s)}-\frac{1}{E[K(s)]})  dN(s) \Big|> \epsilon \Big)\\
 \leq &~ \pr\Big(\frac{2}{nM^2}\underset{t\leq \tau}{\sup} |E[K(t)]-K(t)|>\epsilon \Big)\\
 \leq &~ 8(n+2)\exp \Big\{ -\frac{n\min(\epsilon^2 M^4,M^2)}{128} \Big\},
\end{align*}
for any $n$ satisfying $\frac{1}{n}<\min(\frac{\epsilon^2}{2},\frac{\epsilon^2 M^4}{4})$. This completes the proof. $\Box$\\

\noindent {\bf Proof of Lemma~\ref{lemma:A3}.} For any $t\leq \tau$, we utilize integration by parts to obtain
\begin{align*}
&\Big|\int_0^t \frac{1}{EK(s)}d\{ N(s)-E[N(s)]\}\Big|\\
=&~ \Big| \frac{N(s)-E[N(s)]}{E[K(s)]}\big|_0^t - \int_0^t \{N(s)-E[N(s)] \}d\big\{\frac{1}{E[K(s)]}\big\} \Big|\\
\leq &~ 2 \sup_{t\leq \tau} \big|N(t)-E[N(t)]\big| \frac{1}{E[K(\tau)]}+\sup_{t\leq \tau} \big|N(t)-E[N(t)]\big| \int_0^\tau d\big\{\frac{1}{E[K(s)]}\big\}\\
\leq &~ \frac{3}{M} \sup_{t\leq \tau} \frac{1}{n} \big|N(t)-E[N(t)]\big|.
\end{align*}
Thanks to Lemma \ref{lemma:A1}, we now have
\begin{align*}
 &~\pr\Big(\sup_{t\leq \tau} \Big|\int_0^t  \frac{d\{N(s)-E[N(s)]\}}{E[K(s)]}\Big| > \epsilon\Big)\\
 \leq &~ \pr\Big(\frac{3}{nM}\sup_{t\leq \tau} \big|N(t)-E[N(t)]\big| \Big)\\
 \leq &~ 8(n+1)\exp \big\{-\frac{n\epsilon^2M^2}{288}\big\},
\end{align*}
where $n$ satisfies $\frac{1}{n}\leq \frac{\epsilon^2}{2}$. This completes the proof. $\Box$\\

\section{} \label{proof:thm:bound}

{\bf Preliminary.} The proof of Theorem~\ref{thm:bound} uses two main mechanisms: the concentration bound results we established in Theorem \ref{thm:survbound} to bound the variations in each terminal node, and a construction of a parsimonious set of rectangles, namely $\mathcal{R}$, defined in \citep{wager2015adaptive}. We first introduce some notation. Denote the rectangles $\Noder \in [0,1]^d$ by
\begin{align*}
\Noder=\bigotimes_{j=1}^{d}[r_j^-,r_j^+], \quad \text{where} \quad 0\leq r_j^-<r_j^+\leq 1 \quad \text{for all} \quad j=1,\cdots,d.
\end{align*}
The Lebesgue measure of rectangle $\Noder$ is $\lambda(\Noder)=\prod_{j=1}^d(r_j^+-r_j^-)$. Here we define the expected fraction of training samples and the number of training samples inside $R$, respectively, as follows:
$$\mu(\Noder)=\int_\Noder f(x) dx, \#\Noder=|\{i:X_i\in \Noder\}|.$$
We define the support of rectangle $\Noder$ as $S(\Noder)=\{j\in 1,\ldots,d: r^-_j\neq 0~ \text{or}~ r^+_j\neq 1\}$.

Lemma \ref{lemma:tau} below shows that with high probability there are enough observations larger than or equal to $\tau$ on the rectangle $\Noder$.
\begin{lemma}\label{lemma:tau}
Provided Assumption \ref{asm:tau} holds, the number of observations larger than or equal to $\tau$ on all $\Noder \in \mathcal{R}$ is larger than $\Big(1-\sqrt{\frac{4\log(|\mathcal{R}|\sqrt n)}{kM}}\Big)k M$ with probability larger than $1-1/\sqrt{n}$.
\end{lemma}

\begin{proof}
For one $\Noder \in \mathcal{R}$, by the Chernoff bound, with probability larger than $1-\exp\big\{-\frac{c^2 \#\Noder M}{2})\big\}\geq 1-\exp\big\{-\frac{c^2 k M}{4})\big\}$, the number of observations larger than or equal to $\tau$ on $\Noder$ is larger than $(1-c)k M$, where $0<c<1$ is a constant. Thus with probability larger than $1-1/\sqrt{n}$, the number of observations larger than or equal to $\tau$ on every $\Noder \in \mathcal{R}$ is larger than $\Big(1-\sqrt{\frac{4\log(|\mathcal{R}|\sqrt n)}{kM}}\Big)k M$.
\end{proof}

\noindent {\bf Proof of Theorem~\ref{thm:bound}.} We first establish a triangle inequality by picking some element $\Noder$ in the set $\mathcal{R}$ such that it is a close approximation of $\Node$ and $\Noder\subseteq \Node$. 
\begin{align}
\sup_{t<\tau,\,\Node \in\mathcal{A},\,\mathcal{A}\in \cV} &~ \big\lvert  \hLambda_{\Node,n}(t)-\Lambda_{\Node, n}^\ast(t) \big\lvert \nonumber \\
\leq &~ \sup_{t<\tau,\,\Node \in \mathcal{A},\,\mathcal{A}\in \cV} \inf_{\Noder \in \mathcal{R}} \big\lvert \hLambda_{\Node,n} (t)- \hLambda_{\Noder}(t) \big\lvert \nonumber\\
+ &~  \sup_{t<\tau,\,\Noder \in \mathcal{R},\,\# \Noder \geq k/2} \big\lvert \hLambda_{\Noder,n}(t)-\Lambda_{\Noder, n}^\ast(t) \big\lvert \nonumber\\
+ &~ \sup_{t<\tau,\,\Node \in\mathcal{A},\,\mathcal{A}\in \cV}\inf_{\Noder \in \mathcal{R}} \big\lvert \Lambda_{\Noder, n}^\ast(t) - \Lambda_{\Node, n}^\ast(t)\big\lvert. \label{eqn:conboundthreeparts}
\end{align}
Here, we have $\# \Noder \geq k/2$ in the sub-index of the second term because $\# \Node\geq k$ and from Theorem 10 in \cite{wager2015adaptive}, $ \# \Node - \# \Noder \leq 3\zeta^2 \# \Node/\sqrt{k} +2\sqrt{3\log(|\mathcal{R}|)\# \Node}+O(\log(|\mathcal{R}|))=o(k)$ for any possible $\Node$ with probability larger than $1-1/\sqrt n$. 

We now bound each part of the right hand side of the above inequality. 
Note that we always select a close approximation of $\Node$ from the set $\mathcal{R}$. With a slight abuse of notation, we let the subject index $i$ first run through the observations within $\Noder$ and then through the observations in $\Node$ but not in $\Noder$. This can always be done since $\Noder \subseteq \Node$. Thus we have
\begin{align*}
 &\sup_{t<\tau,\,\Node\in\mathcal{A},\,\mathcal{A}\in \cV} \big| \hLambda_{\Noder,n}(t)-\hLambda_{\Node,n}(t) \big|\\
 \leq & \sup_{t<\tau,\Node \in \mathcal{A},\mathcal{A}\in \cV} \Big| \sum_{s\leq t} \frac{[\Delta N(s)]_{\Noder}}{\sum_{i=1}^{\# \Noder} \mOne(Y_i \geq s)}-\sum_{s\leq t} \frac{[\Delta N(s)]_{\Noder}+[\Delta N(s)]_{\Node \setminus \Noder}}{\sum_{i=1}^{\# \Node} \mOne(Y_i \geq s)} \Big|\\
= & \sup_{t<\tau,\Node \in \mathcal{A},\mathcal{A}\in \cV} \Big| \sum_{s\leq t} \frac{[\Delta N(s)]_{\Noder}}{\sum_{i=1}^{\# \Noder} \mOne(Y_i \geq s)} \\
& \qquad\qquad\qquad\qquad\qquad-\sum_{s\leq t} \frac{[\Delta N(s)]_{\Noder}+[\Delta N(s)]_{\Node \setminus \Noder}}{\sum_{i=1}^{\# \Noder} \mOne(Y_i \geq s)+\sum_{i=\# \Noder+1}^{\# \Node} \mOne(Y_i \geq s)}\Big|\\
\leq & \sup_{t<\tau,\Node \in\mathcal{A},\mathcal{A}\in \cV} \bigg\{ \sum_{j=\# \Noder +1}^{\# \Node} \frac{\Delta N(s_j)}{ \sum_{i=1}^{\# \Noder} \mOne(Y_i \geq s_j)+\sum_{i=\# \Noder+1}^{\# \Node} \mOne(Y_i \geq s_j)  }\\
&\qquad + \sum_{j=1}^{\# \Noder}\Big[ \frac{ \Delta N(s_j)}{\sum_{i=1}^{\# \Noder} \mOne(Y_i \geq s_j)}-\frac{ \Delta N(s_j)}{\sum_{i=1}^{\# \Noder} \mOne(Y_i \geq s_j)+\sum_{i=\# \Noder+1}^{\#\Node} \mOne(Y_i \geq s_j)}\Big] \bigg\},
\end{align*}
where $N(s)=\sum_{i=1}^{n} N_i(s)=\sum_{i=1}^{n} \mOne(Y_i\leq s, \delta_i=1)$. By Lemma \ref{lemma:tau} the first term is bounded by
\begin{align*}
&\frac{\# \Node-\# \Noder}{\Big(1-\sqrt{\frac{4\log(|\mathcal{R}|\sqrt n)}{kM}}\Big)k M} \\
\leq&~ \frac{1}{\Big(1-\sqrt{\frac{4\log(|\mathcal{R}|\sqrt n)}{kM}}\Big)M} \Big[\frac{6\zeta^2}{\sqrt{k}} + 2\sqrt{\frac{6\log(|\mathcal{R}|)}{k}}+O\Big(\frac{\log(|\mathcal{R}|)}{k}\Big)\Big],
\end{align*}
and the second term is bounded by
\begin{align*}
&\sum_{j=1}^{\# \Noder}\Big[ \frac{ \Delta N(s_j)}{\sum_{i=1}^{\# \Noder} \mOne(Y_i \geq s_j)}-\frac{ \Delta N(s_j)}{\sum_{i=1}^{\# \Noder} \mOne(Y_i \geq s_j)+\sum_{i=\# \Noder+1}^{\# \Node} \mOne(Y_i \geq s_j)}\Big]\\
\leq &~ \sum_{j=1}^{\# \Noder} \frac{ \Delta N(s_j) \sum_{i=\# \Noder+1}^{\# \Node} \mOne(Z_i \geq s_j)  }{ \big[\sum_{i=1}^{\# \Noder} \mOne(Y_i \geq s_j)\big] \big[\sum_{i=1}^{\# \Noder} \mOne(Y_i \geq s_j)+\sum_{i=\# \Noder+1}^{\# \Node} \mOne(Y_i \geq s_j)\big] }\\
\leq &~ \sum_{j=1}^{\# \Noder} \frac{ \Delta N(s_j)(\# \Node-\#\Noder)}{ \big[\sum_{i=1}^{\# \Noder} \mOne(Y_i \geq s_j)\big] \big[\sum_{i=1}^{\#\Noder} \mOne(Y_i \geq s_j)+\sum_{i=\# \Noder+1}^{\# \Node} \mOne(Y_i \geq s_j)\big]}\\
\leq &~ \frac{ \#\Noder (\# \Node-\# \Noder)}{ \Big(1-\sqrt{\frac{4\log(|\mathcal{R}|\sqrt n)}{kM}} \Big)^2k^2 M^2}\\
\leq &~ \frac{ 2}{ \Big(1-\sqrt{\frac{4\log(|\mathcal{R}|\sqrt n)}{kM}}\Big)^2 M^2} \Big[\frac{6\zeta^2}{\sqrt{k}} + 2\sqrt{\frac{6\log(|\mathcal{R}|)}{k}}+O\Big(\frac{\log(|\mathcal{R}|)}{k}\Big)\Big].
\end{align*}

Combining these two terms, the first part of Equation \eqref{eqn:conboundthreeparts} is bounded by
\begin{align}
\frac{ 3}{ \Big(1-\sqrt{\frac{4\log(|\mathcal{R}|\sqrt n)}{kM}}\Big)^2 M^2}\Big[\frac{6\zeta^2}{\sqrt{k}} + 2\sqrt{\frac{6\log(|\mathcal{R}|)}{k}}+O\Big(\frac{\log(|\mathcal{R}|)}{k}\Big)\Big],
\label{partI}
\end{align}
with probability larger than $1-1/\sqrt n$. For the second part, by Theorem \ref{thm:survbound},
\begin{align}
 \sup_{t<\tau, \Noder \in \mathcal{R}, \# \Noder \geq k/2} \big|  \hLambda_{\Noder,n}(t)-\Lambda_{\Noder, n}^\ast(t) \big| \leq \frac{\{1728 \log (n)\}^{1/2}}{k^{1/2}M^2},
\label{partII}
\end{align}
with probability larger than $ 1-1/\sqrt n$. The third part of Equation \eqref{eqn:conboundthreeparts} is bounded by
\begin{align}
\notag &\sup_{t<\tau, \Node \in \mathcal{A},\mathcal{A}\in \cV} \big|\Lambda_{\Node, n}^\ast(t)-\Lambda_{\Noder, n}^\ast(t)\big|\\
\notag \leq &\sup_{t<\tau, \Node \in \mathcal{A},\mathcal{A}\in \cV} \Big| \int_0^t \frac{\sum_{i=1}^{\# \Node} \{1-G_i(s)\}dF_i(s)}{\sum_{i=1}^{\# \Node}\{1-G_i(s)\}\{1-F_i(s)\}} -\int_0^t\frac{\sum_{i=1}^{\# \Noder} \{1-G_i(s)\}dF_i(s)}{\sum_{i=1}^{\#\Noder}\{1-G_i(s)\}\{1-F_i(s)\}} \Big|\\
\notag \leq & \sup_{t<\tau,\Node \in \mathcal{A},\mathcal{A}\in \cV}   \int_0^t \bigg| \frac{  \big[\sum_{i=1}^{\# \Noder} \{1-G_i(s)\}dF_i(s) \big]\big[ \sum_{i=\# \Noder+1}^{\# \Node}\{1-G_i(s)\}\{1-F_i(s)\} \big]}{ \big[ \sum_{i=1}^{\# \Noder}\{1-G_i(s)\}\{1-F_i(s)\} \big] \big[\sum_{i=1}^{\# \Node}\{1-G_i(s)\}\{1-F_i(s)\} \big]} \\
\notag & \qquad \qquad \qquad - \frac{\big[\sum_{i=\# \Noder+1}^{\# \Node} \{1-G_i(s)\}dF_i(s) \big]\big[ \sum_{i=1}^{\# \Noder}\{1-G_i(s)\}\{1-F_i(s)\} \big] }{\big[ \sum_{i=1}^{\# \Noder}\{1-G_i(s)\}\{1-F_i(s)\}\big] \big[ \sum_{i=1}^{\# \Node}\{1-G_i(s)\}\{1-F_i(s)\}\big] } \bigg| \\
\notag \leq & \sup_{t<\tau,\Node \in \mathcal{A},\mathcal{A}\in \cV}  \tau \frac{ \#\Node(\#\Node-\#\Noder)}{\#\Noder\# \Node M^4}  \\
\leq & \frac{2\tau}{M^4}\{\frac{3\zeta^2}{\sqrt{k}} +2\sqrt{\frac{3\log(|\mathcal{R}|)}{k}}+O(\frac{\log(|\mathcal{R}|)}{k})\}.
\label{partIII}
\end{align}

Combining inequalities \eqref{partI}, \eqref{partII} and \eqref{partIII} and Corollary 8 in \cite{wager2015adaptive}, we obtain the desired adaptive concentration bound. With probability larger than $1-2/\sqrt n$, we have
\begin{align*}
& \sup_{t<\tau, \Node \in \mathcal{A},\mathcal{A}\in \cV} |\hLambda_{\Node,n}(t)-\Lambda_{\Node, n}^\ast(t)| \\
\leq &~ \frac{ 3}{ \Big(1-\sqrt{\frac{4\log(|\mathcal{R}|\sqrt n)}{kM}}\Big)^2 M^2}\Big[\frac{6\zeta^2}{\sqrt{k}} + 2\sqrt{\frac{6\log(|\mathcal{R}|)}{k}}+O\Big(\frac{\log(|\mathcal{R}|)}{k}\Big)\Big]  \\
&\qquad + \frac{(1728 \log n)^{1/2}}{k^{1/2}M^2}+  \frac{2\tau}{M^4}\Big\{\frac{3\zeta^2}{\sqrt{k}}
+2\sqrt{\frac{3\log(|\mathcal{R}|)}{k}}+O\Big(\frac{\log(|\mathcal{R}|)}{k}\Big)\Big\}\\
\leq &~ M_1 \Big[\sqrt {\frac{\log(|\mathcal{R}|)}{k} }+\sqrt{\frac{\log(n)}{k}} \Big]  \leq M_1 \sqrt {  \frac{\log(n/k)[\log(dk)+\log\log(n)]}{k\log((1-\alpha)^{-1})} },
\end{align*}
where $M_1$ is an universal constant. This completes the proof of Theorem \ref{thm:bound}. $\Box$\\ 

\begin{corollary}\label{thm:boundforest}
Suppose Assumptions \ref{asm:tau}-\ref{asm:k} hold. Then all valid forests concentrate on the censoring contaminated forest with probability larger than $1-2/\sqrt n$,
\begin{align*}
\sup_{t<\tau,\,x\in[0,1]^d,\,\{\ccA_{(b)}\}_1^B \in \cH_{\alpha,k}(\cD_n)} & \left\lvert\hLambda_{\{\ccA_{(b)}\}_1^B,n}(t \mid x)-\Lambda_{\{\ccA_{(b)}\}_1^B, \,n}^\ast(t \mid x)\right\lvert \\
\leq&~ M_1 \sqrt {  \frac{\log(n/k)[\log(dk)+\log\log(n)]}{k\log((1-\alpha)^{-1})} },
\end{align*}
for some universal constant $M_1$. 
\end{corollary}

\noindent {\bf Proof of Corollary~\ref{thm:boundforest}.} Since for any $\ccA \in \cV_{\alpha,k}(\cD_n)$ we have
\begin{align*}
 \sup_{t<\tau,\,x\in [0,1]^d} \left\lvert \hLambda_{\ccA,n}(t\mid x)-\Lambda_{\ccA, n}^\ast(t\mid x) \right\rvert
 \leq ~ M_1\sqrt {  \frac{\log(n/k)[\log(dk)+\log\log(n)]}{k\log((1-\alpha)^{-1})} },
\end{align*}
and furthermore if $\underset{n\rightarrow \infty}{\liminf} (d/n)\rightarrow \infty$, for any $\ccA \in \cV_{\alpha,k}(\cD_n)$,
\begin{align*}
\sup_{t<\tau,\,x\in [0,1]^d}  \left\lvert \hLambda_{\ccA,n}(t\mid x)-\Lambda_{\ccA, n}^\ast(t\mid x) \right\rvert
 \leq ~ M_1\sqrt {  \frac{\log(n)\log(d)}{k\log((1-\alpha)^{-1})} }.
\end{align*}

By the definition of $\cH_{\alpha,k}(\cD_n)$, any $\{\ccA_{(b)}\}_1^B$ belonging to $\cH_{\alpha,k}(\cD_n)$ is an element of $\cV_{\alpha,k}(\cD_n)$. Hence we have,
\begin{align*}
\sup_{t<\tau,\,x\in[0,1]^d,\,\{\ccA_{(b)}\}_1^B \in \cH_{\alpha,k}(\cD_n)} & \left\lvert\hLambda_{\{\ccA_{(b)}\}_1^B,n}(t \mid x)-\Lambda_{\{\ccA_{(b)}\}_1^B, \,n}^\ast(t \mid x)\right\lvert \\
\leq&~ M_1 \sqrt {  \frac{\log(n/k)[\log(dk)+\log\log(n)]}{k\log((1-\alpha)^{-1})} },
\end{align*}
for some universal constant $M_1$. $\Box$\\

\section{} \label{proof:consistency1}
\noindent {\bf Proof of Theorem~\ref{treeconsistency1}.}
In order to show consistency, we first show that each terminal node is small enough in all $d$ dimensions. Let $m$ be the lower bound of the number of splits on the terminal node $\Node$ containing $x$, and $m_i$  be the number of splits on the $i$-th dimension. Then we have
\begin{align*}
n\alpha^m=k, \quad m=\log _{1/\alpha}(n/k)=\frac{\log n-\log k}{\log (1/\alpha)}~~ \text{and} ~~\sum_{i=1}^{d} m_i=m.
\end{align*}
The lower bound of the number of splits on the $i$-th dimension $m_i$ has distribution $Binomial(m,\frac{1}{d})$. By the Chernoff bound on each dimension,
\begin{align*}
\pr\Big(m_i>\frac{(1-c_2)m}{d}\Big) > 1- \exp \Big\{-\frac{c_2^2m}{2d}\Big\}
\end{align*}
with any $0<c_2<1$. Then, by Bonferroni,
\begin{align*}
\pr\Big(\min m_i>\frac{(1-c_2)m}{d}\Big)>1- d \exp \Big\{ -\frac{c_2^2m}{2d} \Big\}.
\end{align*}



Suppose we are splitting at the $i$-th dimension on a specific internal node with $\nu$ observations. 
Recall the splitting rule is choosing the splitting point randomly between the $\max((k+1),\lceil \alpha \nu \rceil)$-th, and $\min( (n-k-1), \lfloor (1-\alpha) \nu \rfloor )$-th observations. Without loss of generality, we consider splitting between the $\lceil \alpha \nu \rceil$-th and $\lfloor (1-\alpha) \nu \rfloor$-th observations. The event that the splitting point is between $q_\alpha$ and $q_{1-\alpha}$ happens with probability larger than $c_3$, where $q_\alpha$ is the $\alpha$-th quantile of the $i$-th component of $X$ conditional on the current internal node and previous splits.
Here $c_3=(1-2\alpha)/8$ and is just a lower bound. Since with probability larger than 1/4, the $\lfloor \frac{\alpha+0.5}{2}\nu \rfloor$-th order statistic is larger than $\alpha$ and the $\lceil \frac{1.5-\alpha}{2}\nu \rceil$-th order statistic is less than $1-\alpha$ for large enough $\nu$, where $\nu$ is known to be larger than $2k$. So with probability larger than $c_3$, the splitting point is between $q_{\alpha}$ and $q_{1-\alpha}$. Thanks to Assumption~\ref{asm:density}, conditioning on the current internal node and previous splits, $\max_{x^{(j)}} p(x^{(j)})/\min_{x^{(j)}} p(x^{(j)}) <\zeta^2$, where $p(x^{(j)})$ is the marginal distribution of $x^{(j)}$.  So with probability larger than $c_3$, the splitting point falls into the interval $[\alpha/\zeta^2, 1-\alpha/\zeta^2]$.


The number of splits which partition the parent node to two child nodes with proportion of length between both $\alpha$ and $1-\alpha$ on the $i$-th dimension of the terminal node  $\Node$ is denoted by $m^\ast$ and is $Binomial(m_i,c_3)$. By the Chernoff bound, for any $0<c_4<1$,
\begin{align*}
\pr\big(m^\ast\geq (1-c_4)c_3m_i \big) \geq 1-\exp\Big\{-\frac{c_4^2c_3 m_i}{2}\Big\}.
\end{align*}
If we denote the length of the $i$-th dimension on the terminal node $\Node$ as $l_i$,
\begin{align*}
\pr\big(l_i\leq (1-\alpha/\zeta^2)^{(1-c_4)c_3m_i}\big) \geq 1-\exp\Big\{-\frac{c_4^2c_3 m_i}{2}\Big\}.
\end{align*}
Furthermore, by combining the $d$ dimensions together, we obtain
\begin{align*}
\pr\big(\max_i l_i\leq (1-\alpha/\zeta^2)^{(1-c_4) c_3 \min_i m_i}\big) \geq 1-d\exp\Big\{-\frac{c_4^2c_3 \min_i m_i }{2}\Big\},
\end{align*}
and then
\begin{align*}
\max_{x_1,x_2\in \Node}||x_1-x_2||\leq \sqrt{d} (1-\alpha/\zeta^2)^{\frac{c_3(1-c_4)(1-c_2)m}{d}},
\end{align*}
with probability larger than $1- d \exp\big\{-\frac{c_2^2m}{2d}\big\} -d \exp\big\{- \frac{(1-c_2)c_3c_4^2m}{2d}\big\}$. Hence, for any observation $x_j$ inside the node $\Node$ containing $x$, by Assumption \ref{asm:lips}, we have
\begin{align*}
\sup_{t<\tau}| F(t\mid x)- F(t\mid x_j)| &\leq L_1 \sqrt{d} (1-\alpha/\zeta^2)^{\frac{c_3(1-c_4)(1-c_2)m}{d}},\\
\sup_{t<\tau}| f(t\mid x)-f(t\mid x_j)|  &\leq (L_1^2+L_2) \sqrt{d} (1-\alpha/\zeta^2)^{\frac{c_3(1-c_4)(1-c_2)m}{d}},
\end{align*}
where $f(\cdot\mid x)$ and $F(\cdot\mid x)$, respectively, denote the true density function and distribution function at $x\in \Node$. Then $\Lambda_{\Node, n}^\ast(t)$ has the upper and lower bounds
\begin{align*}
\int_0^t  \frac{f(s\mid x)+b_1}{1-F(s \mid x)-b_2} \, ds \quad \text{and} \quad \int_0^t  \frac{f(s\mid x)-b_1}{1-F(s \mid x)+b_2}\,ds,
\end{align*}
respectively, where
\begin{align*}
b_1 &=(L_1^2+L_2) \sqrt{d} (1-\alpha/\zeta^2)^{\frac{c_3(1-c_4)(1-c_2)m}{d}}, \,\text{and} \,\, b_2 =L_1 \sqrt{d} (1-\alpha/\zeta^2)^{\frac{c_3(1-c_4)(1-c_2)m}{d}}.
\end{align*}
Hence, $|\Lambda_{\Node, n}^\ast(t)-\Lambda(t\mid x)|$ has the bound
\begin{align*}
\int_0^t \frac{b_1 (1-F(s\mid x)) +b_2 f(s\mid x)}{(1-F(s\mid x)-b_2)(1-F(s\mid x))}ds
\leq M_2 \tau \sqrt{d} (1-\alpha/\zeta^2)^{\frac{c_3(1-c_4)(1-c_2)m}{d}},
\end{align*}
for any $t<\tau$, where $M_2$ is some constant depending on $L_1$ and $L_2$. Hence, for the terminal node $\Node$ containing $x$, we bound the bias by
\begin{align*}
\sup_{t<\tau} |\Lambda^\ast_{\Node, n}(t)-\Lambda(t\mid x)| \leq M_2 \tau \sqrt{d} (1-\alpha/\zeta^2)^{\frac{c_3(1-c_4)(1-c_2)m}{d}},
\end{align*}
with probability larger than $1- d \exp\big\{-\frac{c_2^2m}{2d}\big\} -d\exp\big\{- \frac{(1-c_2)c_3c_4^2m}{2d}\big\}$. Combining this with the adaptive concentration bound result from Theorem \ref{thm:bound}, for each $x$, we further have
\begin{align*}
\sup_{t<\tau} |\hLambda_{\ccA,n}(t\mid x)-\Lambda (t\mid x)|=O\Big (\sqrt {  \frac{\log(n/k)[\log(dk)+\log\log(n)]}{k\log((1-\alpha)^{-1})} }+\big(\frac{k}{n}\big)^{\frac{c_1}{d}}\Big ),
\end{align*}
with probability larger than $1-w_n$, where
\begin{align*}
w_n=\frac{2}{\sqrt{n}}+d \exp\Big\{-\frac{c_2^2\log_{1/\alpha}(n/k)}{2d}\Big\}+d \exp\Big\{- \frac{(1-c_2)c_3c_4^2 \log_{1/\alpha}(n/k)}{2d}\Big\},
\end{align*}
and $c_1=\frac{c_3(1-c_2)(1-c_4)}{\log_{1-\alpha}(\alpha)}$. This completes the proof of point-wise consistency.
$\Box$\\

\noindent {\bf Proof of Theorem~\ref{survivalforestcons1.1}.} From Theorem \ref{treeconsistency1}, we need to establish the bound of $|\hLambda_{\ccA,n}(t\mid x)-\Lambda (t\mid x)|$ under an event with small probability $w_n$. Noticing that $\hLambda_{\ccA,n}(t\mid x)$ is simply the Nelson-Aalen estimator of the CHF with at most $k$ terms, for any $t<\tau$, we have
\begin{align*}
\hLambda_{\ccA,n}(t\mid x)\leq \frac{1}{k}+\ldots+\frac{1}{1}=O(\log(k)),
 \end{align*}
which implies that
\begin{align*}
|\hLambda_{\ccA,n}(t\mid x)-\Lambda (t\mid x)| \leq O(\log(k)).
\end{align*}
Then we have
\begin{align*}
&\sup_{t<\tau} E_X \big|\hLambda_n(t\mid X)-\Lambda (t\mid X)\big|\\
=&~ O \Big (\sqrt {  \frac{\log(n/k)[\log(dk)+\log\log(n)]}{k\log((1-\alpha)^{-1})} }+\big(\frac{k}{n}\big)^{\frac{c_1}{d}}+\log(k)w_n \Big ),
\end{align*}
which leads to the following bounds:
\begin{align*}
&\sup_{t<\tau} E_X| \hLambda_{\{\ccA_{(b)}\}_1^B,n}(t\mid X)-\Lambda(t\mid X)| \\
= &\lim_{B \rightarrow \infty} \sup_{t<\tau} E_X|\frac{1}{B}\sum_{b=1}^B \hLambda_{\ccA_{(b)},n} (t\mid X)-\frac{1}{B}\sum_{b=1}^B \Lambda (t\mid X)|\\
\leq& \lim_{B \rightarrow \infty} \frac{1}{B} \sum_{b=1}^B \sup_{t<\tau} E_X| \hLambda_{\ccA_{(b)},n} (t\mid X)-\Lambda (t\mid X)|\\
=&~ O\Big (\sqrt {  \frac{\log(n/k)[\log(dk)+\log\log(n)]}{k\log((1-\alpha)^{-1})} }+\big(\frac{k}{n}\big)^{\frac{c_1}{d}}+\log(k)w_n \Big ). ~\Box
\end{align*}


\section{} \label{proof:consistency2}


\begin{lemma}\label{lemma:sample/expectation0}
Under Assumption~\ref{asm:tau} and assume that the density function of the failure time $f_i(t) = dF_i(t)$ is bounded above by $L$ for each $i$. The difference between $\Lambda_{\Node,n}^\ast(t)$ and $\Lambda_\Node^\ast(t)$ is bounded by
\begin{align*}
\sup_{t<\tau} \big|\Lambda_{\Node,n}^\ast(t)-\Lambda_\Node^\ast(t)\big| \leq \sqrt{\frac{4\tau^2L^2 \log(4\sqrt n)}{M^2 n}},
\end{align*}
with probability larger than $1-1/\sqrt n$.
\end{lemma}

\begin{proof}
By Hoeffding's inequality, we have for each $s\leq t$,
\begin{align*}
&\pr \Big(\,\Big |\frac{1}{n}\sum_{X_i\in \Node} [1-G_i(s)]f_i(s)-E_X\{[1-G(s\mid X)]f(s\mid X)\}\Big |\\
& \qquad\qquad\qquad\qquad\qquad\qquad\qquad\qquad\geq \sqrt{\frac{L^2 \log(4\sqrt n )}{2n}} \,\, \Big) \leq \frac{1}{2\sqrt{n}},
\end{align*}
and
\begin{align*}
& \pr \Big(\,\Big |\frac{1}{n}\sum_{X_i\in \Node} [1-G_i(s)][1 - F_i(s\mid X)]-E_X\{[1-G(s\mid X)][1 - F(s\mid X)]\}\Big |\\
& \qquad\qquad\qquad\qquad\qquad\qquad\qquad\qquad \geq \sqrt{\frac{ \log(4\sqrt n)}{2n}} \,\, \Big) \leq \frac{1}{2\sqrt{n}}.
\end{align*}
After combining the above two inequalities, we have
\begin{align*}
\sup_{t<\tau} \big|\Lambda_{\Node, n}^\ast(t)- \Lambda_\Node^\ast(t)\big| \leq \sqrt{\frac{4\tau^2L^2 \log(4\sqrt n)}{M^2 n}},
\end{align*}
with probability larger than $1-1/\sqrt{n}$.
\end{proof}

\noindent {\bf Proof of Lemma~\ref{lemma:sampleexpectation}.} In a similar way as done for Lemma \ref{lemma:sample/expectation0}, for each $s\leq t$,
{
\small
\begin{align*}
\pr \bigg(\Big|\frac{1}{n}\sum_{X_i\in \Node} [1-G_i(s)]f_i(s)-E_X\{[1-G(s\mid X)]f(s\mid X)\}\Big|&\\
\geq \sqrt{\frac{L^2 \log(4\sqrt n |\Noder|)}{2n}} \,\bigg) \leq \frac{1}{2\sqrt{n}}&,
\end{align*}
}
and
{
\small
\begin{align*}
\pr \bigg(\Big|\frac{1}{n}\sum_{X_i\in \Node} [1-G_i(s)][1 - F_i(s\mid X)]-E_X\{[1-G(s\mid X)][1 - F(s\mid X)]\}\Big|&\\
\geq \sqrt{\frac{ \log(4\sqrt n |\Noder|)}{2n}}\bigg) \leq \frac{1}{2\sqrt{n}}&.
\end{align*}
}
Thus, with probability larger than $1/\sqrt n$,
\begin{align*}
|\Lambda_{\Node,n}^\ast(t)- \Lambda_\Node^\ast(t)| \leq&~  \sqrt{\frac{4\tau^2L^2 \log(4\sqrt n |\Noder|)}{M^2 n}}\\
\leq&~ M_2\sqrt {  \frac{\log(n/k)[\log(dk)+\log\log(n)]}{k\log((1-\alpha)^{-1})} },
\end{align*}
for all $t<\tau$ and all $\Node \in \mathcal{A},\mathcal{A}\in \cV_{\alpha,k}(\cD_n)$, where $M_2$ is some universal constant depending on $L$ and $M$. $\Box$\\

\begin{lemma}\label{lemma:splitcombine2}
Under the marginal screening splitting rule given in Algorithm \ref{alg:consistency2}, uniformly across all internal nodes, we essentially only split on $(d_0+d_1)$ dimensions with probability larger than $1-3/\sqrt{n}$ on the entire tree.
\end{lemma}

\noindent {\bf Proof of Lemma~\ref{lemma:splitcombine2}.} We prove the lemma in two parts. In the first part, we show that the event that a survival tree ever splits on a noise variable has with probability smaller than $3/\sqrt{n}$. In the second part, we prove that if a failure variable is randomly selected and has never been used in the upper level of the tree, then the probability that the proposed survival tree splits on this variable is at least $1-3/\sqrt{n}$.

We start with defining $\Delta^\ast(c)=\max_{t<\tau} \big| \Lambda_{\Node_j^+(c)}^\ast(t)- \Lambda_{\Node_j^-(c)}^\ast(t) \big|$. Then for any noise variable $j$,
\begin{small}
\begin{align*}
\Delta^\ast(c)=&~\max_{t<\tau} \big| \Lambda_{\Node_j^+(c)}^\ast(t) - \Lambda_{\Node_j^-(c)}^\ast(t) \big| \\
=&~\max_{t<\tau} \bigg| \int_0^t \frac{E_{X\in \Node_j^+(c)} [1-G(s\mid X)] \textup{d} F(s\mid X)}{E_{X\in \Node_j^+(c)} [1-G(s\mid X)][1 - F(s\mid X)]} \\
&\qquad \qquad \qquad  -\int_0^t \frac{E_{X\in \Node_j^-(c)} [1-G(s\mid X)] \textup{d} F(s\mid X)}{E_{X\in \Node_j^-(c)} [1-G(s\mid X)][1 - F(s\mid X)]} \bigg| \\
=&~\max_{t<\tau} \bigg| \int_0^t \frac{  \int_{x^{(j)}\geq c}  \int_{\Node^{d-1}} [1-G(s\mid x^{(-j)} )] \textup{d} F(s\mid x^{(-j)})  p(x^{(-j)}|x^{(j)}) p(x^{(j)}) dx^{(-j)} dx^{(j)} }{ \int_{x^{(j)}\geq c}  \int_{\Node^{d-1}} [1-G(s\mid x^{(-j)} )] [1-F(s\mid x^{(-j)} )]  p(x^{(-j)}|x^{(j)}) p(x^{(j)}) dx^{(-j)} dx^{(j)}  } \\
  -& \int_0^t \frac{  \int_{x^{(j)}< c}  \int_{\Node^{d-1}} [1-G(s\mid x^{(-j)} )] \textup{d} F(s\mid x^{(-j)})  p(x^{(-j)}|x^{(j)}) p(x^{(j)}) dx^{(-j)} dx^{(j)} }{ \int_{x^{(j)}< c}  \int_{\Node^{d-1}} [1-G(s\mid x^{(-j)} )] [1-F(s\mid x^{(-j)} )]  p(x^{(-j)}|x^{(j)}) p(x^{(j)}) dx^{(-j)} dx^{(j)}  } \bigg| \\
  =&~\max_{t<\tau} |(I)-(II)|,
\end{align*}
\end{small}
where $\Node^{d-1}$ refers to integrating over $d$ dimensions except variable $j$ on the internal node $\Node$  and $x^{(-j)}$ refers to $d-1$ dimensions of $x$ except the coordinate $j$. Without loss of generality, we assume that $(I)>(II)$ when the maximum is achieved. By Assumption~\ref{asm:cor},
\begin{small}
  $$(I)<\int_0^t \frac{  \int_{x^{(j)}\geq c}  \int_{\Node^{d-1}} [1-G(s\mid x^{(-j)} )] \textup{d} F(s\mid x^{(-j)})  p_\Node(x^{(-j)})\gamma p(x^{(j)}) dx^{(-j)} dx^{(j)} }{ \int_{x^{(j)}\geq c}  \int_{\Node^{d-1}} [1-G(s\mid x^{(-j)} )] [1-F(s\mid x^{(-j)} )]  p_\Node(x^{(-j)})\gamma^{-1} p(x^{(j)}) dx^{(-j)} dx^{(j)}  },$$
$$(II)> \int_0^t \frac{  \int_{x^{(j)}< c}  \int_{\Node^{d-1}} [1-G(s\mid x^{(-j)} )] \textup{d} F(s\mid x^{(-j)})  p_\Node(x^{(-j)})\gamma^{-1} p(x^{(j)}) dx^{(-j)} dx^{(j)} }{ \int_{x^{(j)}< c}  \int_{\Node^{d-1}} [1-G(s\mid x^{(-j)} )] [1-F(s\mid x^{(-j)} )]  p_\Node(x^{(-j)})\gamma p(x^{(j)}) dx^{(-j)} dx^{(j)}  },$$
\end{small}
where $p_\Node(x^{(-j)})$ refers to the marginal distribution of $x^{(-j)}$ on the internal node $\Node$. So $\Delta^\ast(c)$ is further bounded by
 \begin{small}
 \begin{align*}
& \int_0^t \frac{  \gamma \int_{x^{(j)}\geq c}  p(x^{(j)}) dx^{(j)} \int_{\Node^{d-1}} [1-G(s\mid x^{(-j)} )] \textup{d} F(s\mid x^{(-j)})  p_\Node(x^{(-j)})  dx^{(-j)}  }{ \gamma^{-1} \int_{x^{(j)}\geq c} p(x^{(j)})   dx^{(j)} \int_{\Node^{d-1}} [1-G(s\mid x^{(-j)} )] [1-F(s\mid x^{(-j)} )]  p_\Node(x^{(-j)})  dx^{(-j)}   } \\
  -& \int_0^t \frac{ \gamma^{-1} \int_{x^{(j)}< c} p(x^{(j)}) dx^{(j)}  \int_{\Node^{d-1}} [1-G(s\mid x^{(-j)} )] \textup{d} F(s\mid x^{(-j)})  p_\Node(x^{(-j)})  dx^{(-j)}  }{ \gamma \int_{x^{(j)}< c}  p(x^{(j)}) dx^{(j)}  \int_{\Node^{d-1}}  [1-G(s\mid x^{(-j)} )] [1-F(s\mid x^{(-j)} )]  p_\Node(x^{(-j)})  dx^{(-j)}  } \\
 \leq & (\gamma^2-\gamma^{-2}) \Lambda^\ast_\Node(\tau) \leq (\gamma^2-\gamma^{-2} ) \frac{\tau L}{M^2}. \\
\end{align*}
\end{small}
From the adaptive concentration bound result and Lemma \ref{lemma:sampleexpectation}, we have, for an arbitrary $x\in [0,1]^d$ and a valid partition $\mathcal{A} \in \cV_{\alpha,k}(\cD_n)$,
\begin{align*}
\max_{t<\tau} \big| \hLambda_{\ccA,n}(t\mid x)- \Lambda_\ccA^\ast(t\mid x) \big|  \leq
 M_3\sqrt {  \frac{\log(n/k)[\log(dk)+\log\log(n)]}{k\log((1-\alpha)^{-1})} },
\end{align*}
with probability larger than $1-3/\sqrt{n}$, where $M_3=\max(M_1,M_2)$. Hence
\begin{align*}
 \Delta_1(c) \leq (\gamma^2-\gamma^{-2} ) \frac{\tau L}{M^2}+ M_3\sqrt{\frac{\log(n/k)[\log(dk)+\log\log(n)]}{k\log((1-\alpha)^{-1})} },
\end{align*}
with probability larger than $1-3/\sqrt{n}$ uniformly over all possible nodes with at least $2k$ observations and all noise variables. Thus only with probability less than $3/\sqrt{n}$ will the proposed survival tree split on a noise variable.

To prove the second argument, suppose $\Node$ is the current node and $X^{(j)}$ is an important variable. Since we choose the splitting point $\tilde c$ which maximizes $\Delta_1(c)$, the empirical signal is larger than that of $c_0$. Without loss of generality, we consider a cutoff point of $c_0$ and $\ell^+(j, t_0, c_0) \geq \ell^-(j, t_0, c_0)$. Hence we are interested in
\begin{align*}
\Delta^\ast(c_0)=& \max_{t<\tau} \big| \Lambda_{ \Node_j^+(c_0)}^\ast(t)- \Lambda_{\Node_j^-(c_0)}^\ast(t) \big| \\
=& \max_{t<\tau} \bigg| \int_0^t \frac{E_{\Node_j^+(c_0)} [1-G(s\mid X)] \textup{d} F(s\mid X)}{E_{ \Node_j^+(c_0)} [1-G(s\mid X)][1 - F(s\mid X)]}\\
& \qquad \qquad \qquad -\int_0^t \frac{E_{\Node_j^-(c_0)} [1-G(s\mid X)] \textup{d} F(s\mid X)}{E_{\Node_j^-(c_0)} [1-G(s\mid X)][1 - F(s\mid X)]} \bigg|.
\end{align*}
Since $1-G(\tau)$ is bounded away from 0 by our assumption with $1-G(\tau) \geq M$, the above expression can be further bounded below by
\begin{align*}
&~\Delta^\ast(c_0)\\
\geq & M \int_0^{t_0} \frac{E_{\Node_j^+(c_0)} \textup{d} F(s\mid X)}{E_{ \Node_j^+(c_0)} [1 - F(s\mid X)]}- M^{-1} \int_0^{t_0} \frac{E_{\Node_j^-(c_0)} \textup{d} F(s\mid X)}{E_{ \Node_j^-(c_0)}[1 - F(s\mid X)]}\\
= & M \ell^+(j, t_0, c_0)  -  M^{-1} \ell^-(j, t_0, c_0) > \ell.\\
\end{align*}
Then, by the adaptive concentration bound results above, $\Delta^\ast(c_0)$ has to be close enough to $\Delta_1(c_0)$. Thus we have
\begin{align*}
\Delta_1(\tilde c)\geq \Delta_1(c_0)> \ell- M_3\sqrt {  \frac{\log(n/k)[\log(dk)+\log\log(n)]}{k\log((1-\alpha)^{-1})} },
\end{align*}
with probability larger than $1-3/\sqrt{n}$ uniformly over all possible nodes and all signal variables. $\Box$\\

\noindent {\bf Proof of Theorem~\ref{survivalforestcons2}.}
The results follow by Lemma~\ref{lemma:splitcombine2} and Theorem~\ref{treeconsistency1}.

\section{} \label{proof:consistency3}

\noindent {\bf Proof of Theorem~\ref{thm:survivalforestcons3}.} 

The proof essentially requires that uniformly across all internal nodes, we essentially only split on $d_0$ dimensions with probability larger than $1-3/\sqrt{n}$ on the entire tree. We first define $\tLambda_\Node^\ast(t)$ on a node $\Node$,
$$\tLambda_{\Node}^\ast(t) = \int_0^t \frac{E_{X\in\Node} [1-G(s\mid X)] \textup{d} F(s\mid X) /[1-\widehat G(s\mid X)] }{ E_{X\in\Node} [1-G(s\mid X)][1 - F(s\mid X)]/[1-\widehat G(s\mid X)] }.$$

For any noise variables, $\Delta^\ast(c)\leq (\gamma^2-\gamma^{-2} ){\tau L}/{M^2} $ follows the proof of Lemma~\ref{lemma:splitcombine2}. For any censoring but not failure variable $j$, we have that
\begin{small}
\begin{align*}
\Delta^\ast(c)=&~\max_{t<\tau} \big| \tLambda_{\Node_j^+(c)}^\ast(t) - \tLambda_{\Node_j^-(c)}^\ast(t) \big| \\
=&~\max_{t<\tau} \bigg| \int_0^t \frac{E_{\Node_j^+(c)} [1-G(s\mid X)]/[1-\widehat G(s\mid X)] \textup{d} F(s\mid X)}{E_{ \Node_j^+(c)} [1-G(s\mid X)]/[1-\widehat G(s\mid X)] [1 - F(s\mid X)]} \\
&\qquad \qquad \qquad  -\int_0^t \frac{E_{\Node_j^-(c)} [1-G(s\mid X)]/[1-\widehat G(s\mid X)] \textup{d} F(s\mid X)}{E_{\Node_j^-(c)} [1-G(s\mid X)]/[1-\widehat G(s\mid X)] [1 - F(s\mid X)]} \bigg| \\
 =&~\max_{t<\tau} |(I)-(II)|.
\end{align*}
\end{small}
Without loss of generality, we assume that $(I)>(II)$ when the maximum is achieved.  Since $\lim_{n\rightarrow \infty}\pr(\sup_{t<\tau}|\widehat G(t|X=x)-G(t|X=x)|>\epsilon) = 0$ for any $0<\epsilon<M$ and $x$,  the following inequalities $$\frac{1-G(t|X=x)}{1-\widehat G(t|X=x)}>\frac{1-G(t|X=x)}{1-\widehat G(t|X=x)+\epsilon}>\frac{M}{M+\epsilon},$$ $$\frac{1-G(t|X=x)}{1-\widehat G(t|X=x)}<\frac{1-G(t|X=x)}{1-\widehat G(t|X=x)-\epsilon}<\frac{M}{M-\epsilon},$$ hold. Then $\Delta^\ast(c)$ is further bounded by
\begin{small}
\begin{align*}
 & \frac{M+\epsilon}{M-\epsilon} \int_0^t  \frac{  E_{\Node^+_j(c)} \textup{d} F(s\mid x) }{E_{\Node^+_j(c)} [1- F(s\mid x)] }     - \frac{M-\epsilon}{M+\epsilon} \int_0^t \frac{   E_{\Node^-_j(c)} \textup{d} F(s\mid x)  }{  E_{\Node^-_j(c)} [1- F(s\mid x)]  } \\
\leq & [\frac{M+\epsilon}{M-\epsilon}\gamma^2- \frac{M-\epsilon}{M+\epsilon}\gamma^{-2}] \Lambda_\Node(\tau)
\leq [\frac{M+\epsilon}{M-\epsilon}\gamma^2- \frac{ M-\epsilon}{M+\epsilon}\gamma^{-2}]  \frac{\tau L}{M},
\end{align*}
\end{small}
where the first inequality holds from Assumption~\ref{asm:cor2}.

For any failure variable $j$, 
\begin{align*}
\Delta^\ast(c_0)=& \max_{t<\tau}  \big| \tLambda_{ \Node_j^+(c_0)}^\ast(t)- \tLambda_{\Node_j^-(c_0)}^\ast(t) \big|\\
\geq & \bigg|  \int_0^{t_0} \frac{E_{ \Node_j^+(c_0)} [1-G(s\mid X)]/[1-\widehat G(s\mid X)] \textup{d} F(s\mid X)}{E_{\Node_j^+(c_0)} [1-G(s\mid X)]/[1-\widehat G(s\mid X)] [1 - F(s\mid X)]}\\
& \qquad \qquad \qquad -\int_0^{t_0} \frac{E_{ \Node_j^-(c_0)} [1-G(s\mid X)]/[1-\widehat  G(s\mid X)] \textup{d} F(s\mid X)}{E_{\Node_j^-(c_0)} [1-G(s\mid X)]/[1-\widehat G(s\mid X)] [1 - F(s\mid X)]} \bigg| .
\end{align*}
Without loss of generality, we assume that $\ell^+(j, t_0, c_0) > \ell^-(j, t_0, c_0)$. We have that
\begin{align*}
&~\Delta^\ast(c_0)\\
\geq &  \int_0^{t_0}\frac{E_{\Node_j^+(c_0)} [1-G(s\mid X)]/[1-G(s\mid X)+\epsilon] \textup{d} F(s\mid X)}{E_{\Node_j^+(c_0)} [1-G(s\mid X)]/[1- G(s\mid X)-\epsilon] [1 - F(s\mid X)]}\\
& \qquad \qquad \qquad - \int_0^{t_0}  \frac{E_{ \Node_j^-(c_0)}[1-G(s\mid X)]/[1-  G(s\mid X)-\epsilon] \textup{d} F(s\mid X)}{E_{\Node_j^-(c_0)}[1-G(s\mid X)]/[1-G(s\mid X)+\epsilon] [1 - F(s\mid X)]} \, \\
\geq &  \frac{M-\epsilon}{M+\epsilon} \int_0^{t_0}   \frac{E_{\Node_j^+(c_0)} \textup{d} F(s\mid X)}{E_{\Node_j^+(c_0)}  [1 - F(s\mid X)]} - \frac{M+\epsilon}{M-\epsilon} \int_0^{t_0}  \frac{E_{\Node_j^-(c_0)}\textup{d} F(s\mid X)}{E_{\Node_j^-(c_0)} [1 - F(s\mid X)]} \, \\
=&  \frac{M-\epsilon}{M+\epsilon} \ell^+(j, t_0, c_0) - \frac{M+\epsilon}{M-\epsilon}  \ell^-(j, t_0, c_0), \,
\end{align*}
with probability going to 1. Combined with the adaptive concentration bound results,
\begin{align*}
 \max_{t<\tau}  \big| \tLambda_{\ccA,n}(t\mid x)- \tLambda_\ccA^\ast(t\mid x) \big|  \leq
 M_3 \sqrt {  \frac{\log(n/k)[\log(dk)+\log\log(n)]}{k\log((1-\alpha)^{-1})} },
\end{align*}
we have
\begin{align*}
\Delta_2(\tilde c)\geq \Delta_2(c_0)> \ell-o_p(1),
\end{align*}
uniformly over all possible nodes and all signal variables. $\Box$\\

\newpage

\bibliographystyle{acmtrans-ims}
\bibliography{survtrees}

\end{document}